\documentclass[11pt,reqno]{amsart}

\usepackage[utf8]{inputenc}

\usepackage{amsmath}
\usepackage{amssymb}
\usepackage{amsfonts}
\usepackage{mathtools} 

\usepackage[thmtools-compat]{keytheorems}
\makeatletter
\keytheoremset{restate-counters={section}} 
\makeatother

\usepackage{mathrsfs} 
\usepackage{mleftright} 
\mleftright 

\usepackage[x11names]{xcolor}

\usepackage{tikz-cd}

\usepackage[breaklinks,colorlinks=true,linkcolor=blue,citecolor=green!60!black,backref]{hyperref} 
\hypersetup{pdftex}
\usepackage{hypcap}

\usepackage{zref-clever}
\zcsetup{cap=true}

\let\cref\zcref
\newcommand{\Cref}[1]{\zcref[S]{#1}}

\makeatletter
\newcommand*{\namecref}[1]{%
  \zref@extract{#1}{zc@type}%
}
\makeatother

\zcRefTypeSetup{equation}{ 
  Name-sg =, name-sg =, Name-pl =, name-pl =, Name-sg-ab =, name-sg-ab =, Name-pl-ab =, name-pl-ab =
}

\AddToHook{env/thm/begin}{%
    \zcsetup{countertype={thm=theorem}}}
\AddToHook{env/rmk/begin}{%
    \zcsetup{countertype={thm=remark}}}
\AddToHook{env/lem/begin}{%
    \zcsetup{countertype={thm=lemma}}}
\AddToHook{env/cor/begin}{%
    \zcsetup{countertype={thm=corollary}}}
\AddToHook{env/prop/begin}{%
    \zcsetup{countertype={thm=proposition}}}
\AddToHook{env/conj/begin}{%
    \zcsetup{countertype={thm=conjecture}}}
\AddToHook{env/defn/begin}{%
    \zcsetup{countertype={thm=definition}}}
\AddToHook{env/ex/begin}{%
    \zcsetup{countertype={thm=example}}}

\zcRefTypeSetup{cd}{
  Name-sg = Diagram , name-sg = diagram , Name-pl = Diagrams , name-pl = diagrams ,
}
\AddToHook{env/cd/begin}{%
    \zcsetup{countertype={equation=diagram}}}

\zcRefTypeSetup{defnprop}{
  Name-sg = Definition/Proposition , name-sg = definition/proposition , Name-pl = Definitions/Propositions , name-pl = definitions/propositions ,
}
\AddToHook{env/defn-prop/begin}{%
    \zcsetup{countertype={thm=defnprop}}}

\zcRefTypeSetup{conj}{
    Name-sg = Conjecture , name-sg = conjecture , Name-pl = Conjectures , name-pl = conjectures ,
}

\zcRefTypeSetup{claim}{
  Name-sg = Claim , name-sg = claim , Name-pl = Claims , name-pl = claims ,
}

\zcRefTypeSetup{problem}{
  Name-sg = Problem , name-sg = problem , Name-pl = Problems , name-pl = problems ,
}

\usepackage{autonum}

\makeatletter
\AtBeginDocument{%
    \autonum@generatePatchedReferenceCSL{Cref}%
}
\makeatother

\usepackage[shortlabels]{enumitem}

\usepackage{csquotes}
\MakeOuterQuote{"}

\makeatletter
\def\@cite#1#2{[{#1\if@tempswa ,\penalty400~#2\fi}]}
\makeatother

\usetikzlibrary{arrows}
\usetikzlibrary{babel}
\tikzcdset{every label/.append style = {font = \small}}
\tikzset{
    arrowdecor/.style={anchor=south, rotate=90, inner sep=.5mm}
}
\usepackage{rotating}

\makeatletter
\newcommand{\thesislanguage}[2]{%
  \@ifpackageloaded{babel}{\iflanguage{ngerman}{#1}{#2}}{#2}%
}
\makeatother

\newtheorem{thm}{Theorem}[section]
\newtheorem*{thm*}{Theorem}

\newtheorem*{conj*}{Conjecture}

\newtheorem{lem}[thm]{Lemma}
\newtheorem*{lem*}{Lemma}

\newtheorem*{prop*}{Proposition}

\theoremstyle{definition}
\newtheorem{defn}[thm]{Definition}
\newtheorem*{defn*}{Definition}

\newtheorem{defn-prop}[thm]{Definition/Proposition}

\newtheorem*{convention*}{Convention}

\theoremstyle{remark}
\newtheorem{rmk}[thm]{Remark}
\newtheorem*{rmk*}{Remark}

\newtheoremstyle{claim}
  {0.5\topsep} 
  {0\topsep} 
  {} 
  {} 
  {\itshape} 
  {:} 
  {.5em} 
  {} 
\theoremstyle{claim}

\newtheorem*{claim*}{Claim}
\counterwithin*{claim}{thm}



\newtheoremstyle{Normal}{}{}{}{}{\bfseries}{:}{.5em}{}
\theoremstyle{Normal}

\counterwithin*{problem}{section}

\makeatletter
\let\oldendproof\endproof
\def\endproof{\oldendproof\aftergroup\@afterindentfalse\aftergroup\@afterheading}
\makeatother

\newcounter{autoItemCounter}

\newlist{xxthmlist}{enumerate}{1}
\setlist[xxthmlist]{
    label=(\roman{xxthmlisti})\protect\label{item:\theautoItemCounter\thexxthmlisti},
    ref=(\roman{xxthmlisti})
}

\usepackage{aliascnt}
\newaliascnt{cd}{equation}

\makeatletter

\newenvironment{cd*}{%
    $$
    \begin{tikzcd}[sep=large]
}{%
    \end{tikzcd}
    $$\@ignoretrue
}
\makeatother

\makeatletter
\g@addto@macro\bfseries{\boldmath}
\makeatother

\newcommand{\noop}[1]{}

\newcommand{\inj}{\hookrightarrow}
\newcommand{\weakto}{\rightharpoonup}

\AtBeginDocument{
    \renewcommand{\epsilon}{\varepsilon}
}

\renewcommand{\setminus}{\smallsetminus}

\DeclareMathSymbol{\colonrel}{\mathrel}{operators}{"3A}

\makeatletter
\newcommand{\DeclareMathActive}[2]{%
  \expandafter\edef\csname keep@#1@code\endcsname{\mathchar\the\mathcode`#1 }
  \begingroup\lccode`~=`#1\relax
  \lowercase{\endgroup\def~}{#2}%
  \AtBeginDocument{\mathcode`#1="8000 }%
}

\newcommand{\std}[1]{\csname keep@#1@code\endcsname}
\patchcmd{\newmcodes@}{\mathcode`\-\relax}{\std@minuscode\relax}{}{%
  \PackageWarning{preamble_v3}{Failed to patch \string\newmcodes@ for minus mathcode handling}%
}
\AtBeginDocument{\edef\std@minuscode{\the\mathcode`-}}

\DeclareMathSymbol{\origcolon}{\mathpunct}{operators}{"3A}%

\DeclareRobustCommand{\coloneq}{\mathrel{\mathop{\origcolon}\mkern-6mu=}}

\def\@mathcoloneq={\coloneq}

\DeclareMathActive{:}{\@ifnextchar={\@mathcoloneq}{\nobreak\mskip1mu\mathpunct{}\nonscript \mkern-\thinmuskip{\std{:}\mskip6muplus1mu\relax}}}

\makeatother

\expandafter\let\expandafter\originall\csname\encodingdefault\string\l\endcsname
\DeclareRobustCommand*\l
    {\ifmmode\ell\else\expandafter\originall\fi}

\newcommand{\acts}{\raisebox{-0.155em}{\scalebox{0.88}{ \rotatebox[origin=lb]{90}{$\curvearrowleft$}}}\,}

\usepackage{accents}
\DeclareFontFamily{U}{matha}{\hyphenchar\font45}
\DeclareFontShape{U}{matha}{m}{n}{
    <5> <6> <7> <8> <9> <10> gen * matha
    <10.95> matha10 <12> <14.4> <17.28> <20.74> <24.88> matha12
}{}
\DeclareSymbolFont{matha}{U}{matha}{m}{n}
\DeclareFontSubstitution{U}{matha}{m}{n}
\DeclareMathSymbol{\tinsubset}{3}{matha}{"80}

\let\oldbigwedge\bigwedge
\renewcommand{\bigwedge}{\textstyle\oldbigwedge}

\newcommand{\Ric}{\mathrm{Ric}}

\newcommand{\an}{\mathrm{an}}

\DeclareMathOperator{\Aut}{Aut}
\newcommand{\aut}{\mathfrak{aut}}

\DeclareMathOperator{\_div}{div}
\renewcommand{\div}{\_div}

\DeclareMathOperator{\Isom}{Isom}

\DeclareMathOperator{\Lie}{Lie}

\DeclareMathOperator{\tr}{tr}

\renewcommand{\Re}{\operatorname{Re}}
\renewcommand{\Im}{\operatorname{Im}}

\expandafter\let\expandafter\originald\csname\encodingdefault\string\d\endcsname
\DeclareRobustCommand*\d
    {\ifmmode\mathop{}\!d\else\expandafter\originald\fi}

\newcommand{\p}{\mathrlap{~~\text{.}}}
\expandafter\let\expandafter\originalc\csname\encodingdefault\string\c\endcsname
\DeclareRobustCommand*\c
    {\ifmmode\mathrlap{~~\text{,}}\else\expandafter\originalc\fi}

\newcommand{\demph}{\textbf}

\makeatletter
\newcommand{\ovl}[1]{%
  \overline{\raisebox{0pt}[\dimexpr\height+0.6pt\relax]{\m@th$#1$}}%
}
\makeatother

\makeatletter
\let\save@mathaccent\mathaccent
\newcommand*\if@single[3]{%
    \setbox0\hbox{${\mathaccent"0362{#1}}^H$}%
    \setbox2\hbox{${\mathaccent"0362{\kern0pt#1}}^H$}%
    \ifdim\ht0=\ht2 #3\else #2\fi
}
\newcommand*\rel@kern[1]{\kern#1\dimexpr\macc@kerna}
\newcommand*\widebar[1]{\@ifnextchar^{{\wide@bar{#1}{0}}}{\wide@bar{#1}{1}}}
\newcommand*\wide@bar[2]{\if@single{#1}{\wide@bar@{#1}{#2}{1}}{\wide@bar@{#1}{#2}{2}}}
\newcommand*\wide@bar@[3]{%
    \begingroup
    \def\mathaccent##1##2{%
        \let\mathaccent\save@mathaccent
        \if#32 \let\macc@nucleus\first@char \fi
        \setbox\z@\hbox{$\macc@style{\macc@nucleus}_{}$}%
        \setbox\tw@\hbox{$\macc@style{\macc@nucleus}{}_{}$}%
        \dimen@\wd\tw@
        \advance\dimen@-\wd\z@
        \divide\dimen@ 3
        \@tempdima\wd\tw@
        \advance\@tempdima-\scriptspace
        \divide\@tempdima 10
        \advance\dimen@-\@tempdima
        \ifdim\dimen@>\z@ \dimen@0pt\fi
        \rel@kern{0.6}\kern-\dimen@
        \if#31
        \overline{\rel@kern{-0.6}\kern\dimen@\macc@nucleus\rel@kern{0.4}\kern\dimen@}%
        \advance\dimen@0.4\dimexpr\macc@kerna
        \let\final@kern#2%
        \ifdim\dimen@<\z@ \let\final@kern1\fi
        \if\final@kern1 \kern-\dimen@\fi
        \else
        \overline{\rel@kern{-0.6}\kern\dimen@#1}%
        \fi
    }%
    \macc@depth\@ne
    \let\math@bgroup\@empty \let\math@egroup\macc@set@skewchar
    \mathsurround\z@ \frozen@everymath{\mathgroup\macc@group\relax}%
    \macc@set@skewchar\relax
    \let\mathaccentV\macc@nested@a
    \if#31
    \macc@nested@a\relax111{#1}%
    \else
    \def\gobble@till@marker##1\endmarker{}%
    \futurelet\first@char\gobble@till@marker#1\endmarker
    \ifcat\noexpand\first@char A\else
    \def\first@char{}%
    \fi
    \macc@nested@a\relax111{\first@char}%
    \fi
    \endgroup
}
\makeatother

\def\semicolon{;}
\def\applytolist#1{
    \expandafter\def\csname multi#1\endcsname##1{
        \def\multiack{##1}\ifx\multiack\semicolon
            \def\next{\relax}
        \else
            \csname #1\endcsname{##1}
            \def\next{\csname multi#1\endcsname}
        \fi
        \next}
    \csname multi#1\endcsname}

\def\calchar#1{\expandafter\newcommand\csname c#1\endcsname{{\mathcal #1}}}
\applytolist{calchar}QWERTYUIOPLKJHGFDSAZXCVBNM;
\def\bbchar#1{\expandafter\newcommand\csname bb#1\endcsname{{\mathbb #1}}}
\applytolist{bbchar}QWERTYUIOPLKJHGFDSAZXCVBNM;
\def\bfchar#1{\expandafter\newcommand\csname bf#1\endcsname{{\mathbf #1}}}
\applytolist{bfchar}qwertyuiopasdfghjklzxcvbnmQWERTYUIOPLKJHGFDSAZXCVBNM;
\def\sfc#1{\expandafter\newcommand\csname s#1\endcsname{{\sf #1}}}
\applytolist{sfc}QWERTYUIOPLKJHGFDSAZXCVBNM;
\def\fchar#1{\expandafter\newcommand\csname f#1\endcsname{{\mathfrak #1}}}
\applytolist{fchar}qwertyuopasdfghjkzxcvbnmQWERTYUIOPLKJHGFDSAZXCVBNM;
\def\scchar#1{\expandafter\newcommand\csname sc#1\endcsname{{\mathscr #1}}}
\applytolist{scchar}QWERTYUIOPLKJHGFDSAZXCVBNM;

\usepackage[shortlabels]{enumitem}
\usepackage{esint}
\usepackage{mathtools}
\usepackage{amsfonts}
\usepackage{geometry}
\usepackage{amsmath,amssymb}
\usepackage{microtype}

\newcommand{\envalias}[2]{\newenvironment{#1}{\begin{#2}}{\end{#2}}}
\envalias{theorem}{thm}
\envalias{lemma}{lem}
\envalias{definition}{defn}
\envalias{corollary}{cor}
\envalias{proposition}{prop}
\envalias{remark}{rmk}

\numberwithin{equation}{section}
\envalias{conjecture}{conj}

\DeclareMathOperator{\pr}{pr}
\newcommand{\PSH}{\mathrm{PSh}}
\DeclareMathOperator{\MA}{MA}

\usepackage[style=1,leftmargin=0pt,rightmargin=0pt]{mdframed}

\makeatletter
\def\l@section{\@tocline{1}{5pt plus1pt}{0pt}{0pt}{\normalfont\scshape}}
\def\l@subsection{\@tocline{2}{1pt}{1.8pc}{3.2pc}{\normalfont\small}}
\renewcommand{\tocsection}[3]{%
  \indentlabel{\@ifnotempty{#2}{\ignorespaces#1 #2.\quad}}%
  \MakeUppercase{#3}}
\renewcommand{\tocsubsection}[3]{%
  \indentlabel{\@ifnotempty{#2}{\ignorespaces#1 #2.\quad}}#3}
\makeatother

\begin{document}

\title[Matsushima, uniqueness, and K-polystability]
{A Matsushima theorem, uniqueness, and K-polystability for K\"ahler--Ricci shrinkers}
\author{Carlos Esparza}
\address{C.\ N.\ Yang Institute for Mathematics and Fundamental Physics,
1899 Shilian South Road, Hefei, Anhui 230088, China}
\email{ca2010es@gmail.com}
\author{Junsheng Zhang}
\address{Courant Institute of Mathematical Sciences\\
  New York University, 251 Mercer St\\
  New York, NY 10012\\}
\email{jz7561@nyu.edu}
\begin{abstract}
For noncompact smooth K\"ahler--Ricci shrinkers, we prove a Matsushima
theorem identifying the centralizer of the soliton field with the
complexification of its Killing centralizer, uniqueness up to biholomorphism
on a fixed complex manifold, and K-polystability of the associated polarized
Fano fibration.
\end{abstract}
\maketitle
\vspace{-2em}
\tableofcontents

\linespread{1.08}

\section{Introduction}
\label{sec:intro}
Throughout, a Kähler–Ricci shrinker refers to a complete shrinking gradient Kähler–Ricci soliton. In this paper we are concerned exclusively with 
noncompact K\"ahler--Ricci
shrinkers.
 Let $(X^n,\omega,J,f)$ be a noncompact
K\"ahler--Ricci shrinker. We normalize
\[
    \Ric(\omega) + \sqrt{-1}\partial\bar\partial f = \omega,
    \qquad
    \Delta_\omega f + f - |\nabla^{1, 0}f|^2 = 0,
    \qquad
    \rho \coloneq \sqrt{f + n}
    \p
\]
The real holomorphic Killing field $\xi\coloneq J\nabla f$ is called the soliton vector field. We will use the weighted Laplacian with the normalization $\Delta_{\omega, f} = \Delta_\omega - \frac12 (\nabla f)$. When the metric can be understood from context we just write $\Delta_f$ for $\Delta_{\omega, f}$.

\subsection{Main results}
\begin{theorem}[Noncompact Matsushima theorem]
\label{thm:introduction-subcritical-matsushima}
Let $\mathfrak h_\xi$ be the Lie algebra of real holomorphic vector fields commuting with
$\xi$.
Let $\mathfrak k_\xi$ be the Lie algebra of real holomorphic Killing
fields which commute with $\xi$. Then
\[
    \mathfrak h_\xi=\mathfrak k_\xi\oplus J\mathfrak k_\xi
    \p
\]
\end{theorem}

\begin{theorem}[Uniqueness up to biholomorphisms]
\label{thm:introduction-unique-krs}
Let $(\omega_i,f_i,\xi_i=J\nabla^{\omega_i}f_i)$, $i=0,1$, be K\"ahler--Ricci shrinkers on a complex manifold $X$. Then there
exists a biholomorphism $\Psi\in\Aut(X)$ such that
\[
    \Psi^*\omega_1=\omega_0,
    \qquad
    \Psi_*\xi_0=\xi_1
    \p
\]
\end{theorem}

\begin{theorem}[K-polystability]
\label{thm:kpolystability}
Let $(X,\omega,f)$ be a K\"ahler--Ricci
shrinker. Then its associated polarized Fano fibration is
K-polystable in the sense defined in \cite{SunZhangFano}.
\end{theorem}

When \(X\) is compact, the above three results were established in \cite{TianZhuUniqueness,BermanWittNystrom}. Under various curvature assumptions, analogous results were proved in \cite{CDS,EAC,CE}. In this paper, we establish all three results without any curvature assumptions. Uniqueness of the soliton vector field and the Matsushima theorem are proved independently in \cite{CD2026}.
The literature on Kähler--Einstein metrics on Fano manifolds, compact Kähler--Ricci shrinkers, and 
Ricci-flat Kähler cones is far too extensive to be surveyed here; we refer to 
\cite{SunZhangFano} for a brief historical account and apologize for the many contributions 
we are unable to mention. 

The approach of this paper is to connect the algebraic data of the  polarized Fano fibration \cite{SunZhangFano} to the geometry of the shrinker metric, without assuming curvature decay. The fibration supplies finitely many homogeneous holomorphic functions and pluri-anticanonical sections, giving an embedding of the shrinker into am ambient $\mathbb{P}^N \times \mathbb{C}^M$, but it not clear a priori how the shrinker metric is related to these coordinates.
The starting point of our discussion is the sharp polynomial-growth bound for homogeneous holomorphic functions proved by  \cite[Corollary~4.3]{LiZhang2026}, which gives us our first relation between the polarized Fano fibration structure and the metric geometry of the shrinker:
A function $f$ which is a polynomial with respect to the algebraic structure has polynomial growth with respect to the shrinker metric.
Crucially, the degree of the polynomial growth of $f$  is also determined algebraically, as the largest weight of its homogeneous components under the action of soliton vector field.
We combine this estimate with the Schwarz-type lower bound for the shrinker metric to prove (\cref{prop:affine-nondegeneracy}) that there exists a $c>0$ such that 
\[
    |\nabla^{1,0}f|^2=\rho^2-R_\omega\geq c\rho^2
    \quad\text{outside a compact set.}
\]
 This forms
the common analytic input for all three results. 

Together with polynomial control of the homogeneous pluri-anticanonical sections, these estimates yield curvature-free comparisons between the shrinker metrics and an ambient reference Kähler metric on $\mathbb{P}^N \times \mathbb{C}^M$, as well as weighted integral bounds for the shrinker metrics. These bounds provide the weighted integrability needed for Matsushima's theorem and for the uniqueness of the soliton vector field.

It also shows that the canonical anticanonical weights have quadratic growth along the flow of the soliton vector field. This observation both motivates and makes possible the adaptation of the argument in \cite{EAC} to develop a general pluripotential theory for polarized Fano fibrations.

For the proofs of uniqueness of Kähler--Ricci shrinkers and K-polystability, another crucial ingredient is \Cref{prop:general-geodesic}, which guarantees the existence of geodesics within the quadratic-growth class. Once this is established, most of the arguments in \cite{EAC} and \cite{CE} carry over to the general setting without modification. We therefore collect only the results needed here and provide additional details where the arguments require changes.

Independently of the results above, we use the conjugate heat kernel estimates in \cite{LWII} to give an analytic proof that the zero set of the soliton vector field is compact. This provides a new proof of the main result in \cite{SunZhangFano} and strengthens it to a uniform statement using the uniform entropy lower bound for K\"ahler--Ricci shrinkers proved in \cite{Zhang2026}.
\begin{theorem}\label{thm-uniform bound on critical values}
                There exists a constant $C_n$, depending only on the complex dimension $n$, such that $f(x) \leq C_n$ whenever $\nabla f(x) = 0$ for any K\"ahler--Ricci shrinker $(X^n, \omega, J,f)$. 
        \end{theorem}

\subsection{Organization}
\Cref{sec:ambient-reference-comparison} proves the ambient comparison, and \cref{sec:weighted-matsushima} applies it to Matsushima and soliton-field uniqueness. \Cref{sec:pluripotential_theory_on_polarized_fano_fibrations}
develops pluripotential theory for general polarized Fano fibrations,
establishing several convexity properties and the existence of geodesic
segments.  \cref{sec:uniqueness} proves metric uniqueness for a
fixed soliton vector field.
\Cref{sec:universal-k-polystability} constructs the test-configuration ray and proves K-polystability for K\"ahler--Ricci shrinkers.

\subsection*{Acknowledgements}
The authors thank Charles Cifarelli for helpful comments on a draft of this paper. They also thank IGP in Hefei, where part of this work was carried out, for its hospitality. The second author thanks Song Sun for helpful discussions and IASM at Zhejiang University for its hospitality during part of this work.

\subsection*{Declaration on the use of AI}
The main ideas of this paper were developed by the authors. Large language models (LLMs) were used to assist with manuscript preparation and suggested Lemma~\ref{lem:psh-subgeodesic}. In addition, \cref{sec-ambient comparsion II} was developed primarily based on suggestions from LLMs. The authors take full responsibility for the content and correctness of the paper.

\section{Ambient metric comparison}
\label{sec:ambient-reference-comparison}
 Let $(X^n,\omega,J,f)$ be a noncompact
K\"ahler--Ricci shrinker.
Fix the homogeneous ambient system supplied by the Fano-fibration
construction \cite{SunZhangFano}:
\[
    \cA = (h_1, \ldots, h_N, \Psi):X \longrightarrow
    \bbC^N\times\bbP^M
    \p
\]
Here $\cA$ is a globally defined holomorphic embedding. The functions $h_\alpha$ are
homogeneous affine coordinates, so
\[
    \xi h_\alpha = \sqrt{-1}\lambda_\alpha h_\alpha,
    \quad \lambda_\alpha > 0
    \p
\]
We write
\[
    \pi = (h_1, \ldots, h_N):X \longrightarrow Y \subset \bbC^N
\]
for the affine Fano-fibration map. The projective component is
\[
    \Psi = [s_0:\cdots:s_M],
\]
where the sections $s_j \in H^0(X,-\ell K_X)$, for a fixed $\ell \geq 1$, form a finite homogeneous pluri-anticanonical system.
\begin{rmk} \label{rmk:normalLowWt}
    The eigenspaces of $\cL_\xi$ on $H^0(X, -\ell K_X)$ are finite-dimensional, so by adding more sections $s_i$ we may assume that for every $\nu_i$ such that $\cL_\xi s_i = \sqrt{-1} \nu_i s_i$, the eigenspace $\{s \mid \cL_\xi s = \sqrt{-1} \nu_i s\}$ is contained in the span of the $s_i$.
\end{rmk}

The ambient reference form is
\begin{equation} \label{eq:fixed-ambient-reference-form}
    \omega_{\cA} = \mathcal{A}^*(\omega_{\bbC^N}
              + \frac{1}{\ell}\omega_{\rm FS}) =: \omega_{\rm aff} + \frac{1}{\ell}\Psi^*\omega_{\rm FS}
    \p
\end{equation}
The Kähler metric $\omega_\cA$ is the curvature of a hermitian metric $\phi_\cA$ on $K_X^{-1} \cong \Psi^*\mathcal O_{\bbP^M}(1/\ell)$
\begin{equation}\label{eq:fixed-ambient-reference-metric}
    \phi_\cA = \sum_{\alpha=1}^N |z_\alpha|^2 + \frac{1}{\ell}h_\mathrm{FS}
\end{equation}
where $h_\mathrm{FS}$ is the canonical metric on $\cO_{\bbP^M}(1)$ with curvature $\omega_\mathrm{FS}$.

Let
\begin{equation}
    \quad X_0 = -\frac12 J\xi=\frac12 \nabla f, \quad \gamma_s=\exp(sX_0)
    \p
\end{equation}
For $t \in [-1, 0)$, set
\begin{equation} \label{eq:shrinkerRF}
    \omega_t = |t|\gamma_{-\log|t|}^*\omega,
    \qquad
    F(x, t) = |t|\rho^2(\gamma_{-\log|t|}(x))
    \p
\end{equation}
We recall the following identities, which will be used repeatedly throughout
the paper, often without explicit reference:

\[
    R_\omega + |\nabla^{1,0}f|^2 = \rho^2,
    \qquad
    \Delta_f \rho^2 = n - \rho^2
    \p
\]

    \begin{equation} \label{eq:self-similar-identities}
        \partial_t \omega_t = -\Ric(\omega_t),
        \qquad
        \partial_t F = -\gamma_{-\log|t|}^*R_\omega = -|t|R_{\omega_t},
        \qquad
        (\partial_t - \Delta_{\omega_t})F = -n
        \p
    \end{equation}

\subsection{Affine estimates}

We begin with polynomial control of the homogeneous affine coordinates obtained in \cite[Section~4]{LiZhang2026} and
its consequences for the soliton flow.

\begin{lemma}  \label{lem:homogeneous-holomorphic-polynomial-growth}
    Let $h$ be a global holomorphic function satisfying $\xi(h) = \sqrt{-1}\lambda h$, where $\lambda \geq 0$.  Then
    \begin{equation} \label{eq:homogeneous-function-polynomial-pointwise}
        |h| \leq C(1 + \rho)^\lambda,
        \qquad
        |\partial h|_\omega \leq C(1 + \rho^2)^A
    \end{equation}
    for constants $C, A < \infty$ depending on $h$.
\end{lemma}

\begin{proposition}
\label{prop:ambient-lower-bound}
    There are constants $C, A < \infty$ such that
    \begin{equation} \label{eq:ambient-trace-lower-bound}
        \tr_{\omega}\omega_{\rm aff}
        \leq C(1+\rho^2)^A
        \p
    \end{equation}
Moreover, for every $\varepsilon>0$, there is $C_\varepsilon<\infty$ such that
    \begin{equation} \label{eq:parabolic-affine-schwarz}
        \tr_{\omega_t}\omega_{\rm aff}
            \leq C_\varepsilon e^{\varepsilon F}
        \qquad\text{on }X\times[-1, 0)
        \p
    \end{equation}
\end{proposition}

\begin{proof}
    We recall the proof of \cref{eq:parabolic-affine-schwarz} from \cite[Section~4]{LiZhang2026}.
    Put $u = \tr_{\omega_t}\omega_{\rm aff} = \sum_\alpha|\partial h_\alpha|_{\omega_t}^2$.
    Using \cref{lem:homogeneous-holomorphic-polynomial-growth,eq:shrinkerRF} and that $h_\alpha$ has weight $\lambda_\alpha$ along the flow $\gamma_t$, we obtain
    \begin{equation} \label{eq:affine-derivative-scaling}
        u(t, x) = \sum_\alpha |\partial h_\alpha|_{\omega_t}^2(x) = \sum_\alpha |t|^{\lambda_\alpha - 1} |\partial h_\alpha|_\omega^2(\gamma_{-\log|t|}x)
        \p
    \end{equation}
    Fix $t_0\in(-1,0)$.
    Together with $\rho^2(\gamma_{-\log|t|}x)=F(x,t)/|t|$, the equation above gives
    \begin{equation} \label{eq:uPolyBound}
        u \leq C_{t_0}(1 + F)^A
        \qquad\text{on }X\times[-1, t_0]
        \p
    \end{equation}
    Moreover, $\rho^2\circ\gamma_{-\log|t|}\geq\rho^2$, so $F(\cdot,t)\geq|t_0|\rho^2$ on this strip. Thus $F(\cdot, t) \to \infty$ uniformly in $t \in [-1, t_0]$ at spatial infinity.

    Along K\"ahler--Ricci flow, the Bochner formula and Cauchy--Schwarz give respectively
    \[
        (\partial_t - \Delta_{\omega_t})u = -\sum_\alpha|\nabla\partial h_\alpha|_{\omega_t}^2,
            \qquad
        |\partial u|_{\omega_t}^2 \leq u\sum_\alpha|\nabla\partial h_\alpha|_{\omega_t}^2
        \p
    \]
    Hence
    \[
        (\partial_t - \Delta_{\omega_t}) \log(u + 1)
            = \frac{(\partial_t - \Delta_{\omega_t}) u}{u+1} + \frac{|\partial u|_{\omega_t}^2}{(u+1)^2}
            \leq \frac{(\partial_t - \Delta_{\omega_t}) u}{(u+1)^2}
            \leq 0
        \p
    \]
    For $\epsilon>0$, let
    \[
        Q_\varepsilon = \log(u + 1) - \varepsilon F - n\varepsilon t
        \p
    \]
    Then $(\partial_t-\Delta_{\omega_t})Q_\varepsilon \leq 0$.

    \Cref{eq:uPolyBound} shows that $Q_\varepsilon \to -\infty$ uniformly at spatial infinity. Thus by the maximum principle, its maximum is attained at $t=-1$, with a bound independent of $t_0$.  Letting $t_0\nearrow0$ proves
    \cref{eq:parabolic-affine-schwarz}.
\end{proof}

Choose $r_0\gg1$ above all critical values of $\rho^2=f+n$, and set
\begin{equation}
    \Sigma = \{\rho^2 = r_0\}
    \p
\end{equation}
Then the map
\[
    \Gamma:\Sigma\times[0, \infty) \longrightarrow \{\rho^2 \geq r_0\},
    \qquad
    \Gamma(y, s) = \gamma_s(y),
\]
is a diffeomorphism.

\begin{lemma}
\label{lem:affine-return-time}
    There is $C < \infty$ such that, whenever $x = \gamma_s(y)$ with $y \in \Sigma$ and $s \geq 0$,
    \begin{equation} \label{eq:affine-return-time}
        \log\frac{\rho^2(x)}{r_0} \leq s
        \leq \log\frac{\rho^2(x)}{r_0} + C
        \p
    \end{equation}
    Consequently,
    \begin{equation} \label{eq:ScalarIntegralEstimate}
        \int_0^s\frac{R_\omega}{\rho^2}(\gamma_q y)\, dq
        \leq C
        \p
    \end{equation}
\end{lemma}

\begin{proof}
    The first inequality follows from the shrinker equation and the non-negativity of the scalar curvature \cite{ChenScalar}:
    \begin{equation}\label{eq:affine-log-radius-derivative}
    \frac d{dq}\log\rho^2(\gamma_q y)
    = 1 - \frac{R_\omega}{\rho^2}(\gamma_q y)
        \p
    \end{equation}
    Since $\sum_\alpha |h_\alpha|^2 > 0$ on $\Sigma$, there exists $c_0 > 0$ such that for every $y$ we have $\max_\alpha|h_\alpha(y)| \geq c_0$.  Given $y$, choose such an index $\alpha$.  Homogeneity and \cref{lem:homogeneous-holomorphic-polynomial-growth} give
    \[
        c_0 e^{\lambda_\alpha s/2}
        \leq |h_\alpha(x)|
        \leq C_\alpha(1 + \rho(x))^{\lambda_\alpha}
        \p
    \]
    With
    $C=2\max_\alpha\lambda_\alpha^{-1}
    \bigl|\log(C_\alpha/c_0)\bigr|$, we obtain
    \[
        s \leq 2\log(1 + \rho(x)) + C
        \leq \log\frac{\rho^2(x)}{r_0} + C'
        \p
    \]
    Integrating \cref{eq:affine-log-radius-derivative} and using the second inequality in \cref{eq:affine-return-time} gives \cref{eq:ScalarIntegralEstimate}.
\end{proof}

\begin{proposition} \label{prop:affine-nondegeneracy}
    There exists $\theta > 0$ such that
    \begin{equation} \label{eq:radial-gap}
        |\nabla^{1, 0} f|^2 = \rho^2 - R_\omega \geq \theta\rho^2
        \quad\text{on } \{\rho^2 \geq r_0\}
        \p
    \end{equation}
\end{proposition}

\begin{proof}
    Positivity of all affine weights $\lambda_\alpha$ gives
    \[
        c_\Sigma := \min_{y \in \Sigma} \sum_\alpha \lambda_\alpha^2|h_\alpha(y)|^2
            = \min_{y \in \Sigma} \left\vert \nabla^{1,0} f \right\vert^2_{\omega_\mathrm{aff}}(y)
            > 0
        \p
    \]
    For arbitrary $y \in \Sigma, s > 0$ we write $x = \gamma_s(y)$ and $t = - e^{-s}$. Then,
    \[
        c_\Sigma
            \leq \tr_{\omega_t} \omega_\mathrm{aff}(y) \, |\nabla^{1, 0}f|_{\omega_t}^2(y)
            = \tr_{\omega_t} \omega_\mathrm{aff}(y) \, e^{-s}|\nabla^{1,0}f|_\omega^2(x)
        \p
    \]
    By \cref{eq:parabolic-affine-schwarz}, we know that $\tr_{\omega_t} \omega_\mathrm{aff}(y)$ is bounded uniformly for $t\in [-1,0)$.
    Thus $|\nabla^{1,0}f|^2=\rho^2(x) - R_\omega(x) \geq C^{-1} e^s$.  Combining this with $\rho^2(x) \leq r_0 e^s$ proves the desired inequality.
\end{proof}

\begin{remark}
We expect Proposition~\ref{prop:affine-nondegeneracy} to have many further applications. In particular, the volume non-collapsing estimate of \cite[Theorem~7.3]{LiZhang2026} applies to all K\"ahler--Ricci shrinkers with maximal volume growth for which $\xi$ generates an $S^1$-action.
\end{remark}

\subsection{Homogeneous sections and the ambient volume form}

\begin{lemma}\label{lem:homogeneous-pluri-anticanonical-polynomial-growth}
    Let $h_\omega$ be the Hermitian metric on $-K_X$ induced by $\omega^n$.  Suppose that $s \in H^0(X, (-K_X)^\ell)$, $\ell \geq 1$, is homogeneous of finite weight $\nu$ for the canonical lift of the $X_0$ flow. Then there exists $C < \infty$ such that
    \begin{equation} \label{eq:homogeneous-section-polynomial-growth}
        |s|_{h_\omega^\ell}^2 \leq C(1 + \rho^2)^{2\nu+\ell n}
        \p
    \end{equation}
\end{lemma}

\begin{proof}
    The metric $h_\omega$ is induced by the volume form $\omega^n$, and
    \[
        \cL_{X_0}\omega^n = (n - R_\omega)\omega^n
        \p
    \]
    Combining its variation with the weight of $s$ gives
    \begin{equation} \label{eq:homogeneous-section-norm-flow}
        X_0 \log |s|_{h_\omega^\ell}^2 = 2\nu + \ell(n - R_\omega)
    \end{equation}
    on $\{s \neq 0\}$.
    Put $N = 2\nu + \ell n$. Then on $\{f >2N\} \cap \{s \neq 0\}$, we have
    \begin{align} \label{eq:homogeneous-section-polynomial-monotonicity}
        X_0 \log\left(|s|_{h_\omega^\ell}^2f^{-N}\right)
            &= 2\nu + \ell(n - R_\omega) - N\frac{f - R_\omega}{f}\\
            &= 2\nu + \ell n - N - R_\omega\left(\ell - \frac Nf\right) \leq 0
        \p
    \end{align}
    This proves \cref{eq:homogeneous-section-polynomial-growth} on $\{f > 2N\}$; increasing $C$ covers the compact sublevel set.
\end{proof}

\begin{proposition}
    \label{prop:ambient-volume-upper}
    There are $C, A < \infty$, depending only on the fixed homogeneous ambient system and the shrinker, such that on $X$,
    \begin{equation}\label{eq:ambient-volume-upper}
        \omega^n \leq C(1 + \rho^2)^A\omega_{\cA}^n
        \p
    \end{equation}

\end{proposition}

\begin{proof}
    Recall that $\Psi=[s_0:\cdots:s_M]$ is defined by homogeneous sections $s_j\in H^0(X,-\ell K_X)$ with weights $\nu_j$.
    Since
    \[
        \left(\frac{s_i}{s_j}\right) \circ \gamma_s = e^{(\nu_i - \nu_j)s}\frac{s_i}{s_j},
    \]
    we have $\Psi\circ\gamma_s = A_s \circ \Psi$, where
    \[
        A_s[w_0:\cdots:w_M] = [e^{\nu_0 s}w_0:\cdots:e^{\nu_M s}w_M]
        \p
    \]
    Let $V$ be the vector field on $\bbP^M$ generating the 1-parameter subgroup $A_s$.
    Compactness of $\bbP^M$ gives $-C\omega_{\rm FS} \leq \cL_V \omega_{\rm FS} \leq C\omega_{\rm FS}$, while
    \[
        \cL_{X_0}\omega_{\rm aff} = \sqrt{-1}\sum_\alpha\lambda_\alpha \partial h_\alpha\wedge\bar\partial\overline{h_\alpha}
        \p
    \]
    Consequently,
    \begin{equation} \label{eq:ambient-reference-flow-bound}
        -C_{\cA}\omega_{\mathcal A} \leq \cL_{X_0}\omega_{\mathcal A} \leq C_{\cA}\omega_{\mathcal A}
    \end{equation}
    and we obtain
    \[
        X_0 \log\left(\frac{\omega^n}{\omega_{\mathcal{A}}^n}\right)
            = n - R_\omega - \tr_{\omega_{\cA}} (\cL_{X_0}\omega_{\mathcal A})
            \leq A
        \p
    \]
    If $x = \gamma_s(y)$, $y \in \Sigma$, then compactness of $\Sigma$ and \cref{eq:affine-return-time} give
    \[
        \frac{\omega^n}{\omega_{\cA}^n}(x) \leq e^{As} \max_\Sigma\frac{\omega^n}{\omega_{\mathcal A}^n}
                                 \leq C(1 + \rho^2(x))^A
        \p
    \]
    By increasing $C$ we can assure that this equation also holds for $x$ in the compact sublevel set $\{\rho^2 \leq r_0\}$.
\end{proof}

\begin{lemma}
    \label{lem:projective-denominator}
    Let
    $
    S = \sum_{j = 0}^M|s_j|_{h_\omega^\ell}^2.
    $
    Then there are $C, A < \infty$ such that for $x=\gamma_sy$ with $y\in \Sigma$ and $s\geq 0$,
    \begin{equation} \label{eq:projective-denominator-curvature-free-growth}
      C^{-1}(1+\rho(x))^{-A}e^{-\ell\int_0^s R_\omega(\gamma_\tau(y))\,d\tau}
        \leq S(x) \leq C(1 + \rho^2(x))^A
        \p
    \end{equation}
    In particular,
    \begin{equation} \label{eq:projective-log-curvature-free-growth}
        |\log S| \leq C(1 + \rho^2)
        \p
    \end{equation}
\end{lemma}

\begin{proof}
    The upper bound follows by summing \cref{lem:homogeneous-pluri-anticanonical-polynomial-growth} over the
    finite system.  Base-point-freeness and compactness give
    \[
        \min_\Sigma S = c_0 > 0
        \p
    \]
    Homogeneity of $s_j$ gives
    \[
        X_0 \log S \geq -C - \ell R_\omega
        \p
    \]
    If $x=\gamma_s(y)$ with $y\in\Sigma$, then $0 \leq R_\omega \leq \rho^2$ and $\rho^2(\gamma_\tau y) \leq r_0e^\tau$.  Integration gives
    \[
        \log S(x)
            \geq \log c_0 - Cs - \ell\int_0^sR_\omega(\gamma_\tau(y)) \d \tau
            \geq -C - C s - \ell\int_0^s r_0 e^\tau \, d \tau
            \geq -C - C \rho^2
        \p
    \]
    Exponentiating proves the lower bound in
    \cref{eq:projective-denominator-curvature-free-growth}.
\end{proof}

\subsection{An ambient exhaustion and weighted integrability}

The affine coordinates now provide a proper exhaustion whose derivatives and
sublevel volumes are controlled in the ambient metric. This supplies the
cutoffs for the weighted reverse-trace argument.

\begin{lemma} \label{lem:ambient-polynomial-exhaustion}
    Put $H_{\mathrm{aff}}=\sum_{\alpha=1}^N|h_\alpha|^2$.
    There are $0 < \nu \leq 1$ and $C,D<\infty$ such that, for
    \[
        G=(1+H_{\mathrm{aff}})^\nu,
    \]
    one has
    \begin{equation}\label{eq:ambient-G-f-comparison}
        C^{-1}(1 + \rho^2)^{\frac1D} \leq G \leq C(1 + \rho^2)^D
        \p
    \end{equation}
    Moreover,
    \begin{equation}\label{eq:ambient-G-derivatives}
        |\partial G|_{\omega_{\mathcal A}}^2\le C(1+G),
            \qquad
        |\Delta_{\omega_{\mathcal A}}G|\le C,
    \end{equation}
    and, for every $R \geq 1$,
    \begin{equation}\label{eq:ambient-G-volume-growth}
        \operatorname{Vol}_{\omega_{\cA}}\{G \leq R\} \leq C(1 + R)^D
        \p
    \end{equation}
\end{lemma}

\begin{proof}
    In fact any $0 < \nu \leq 1$ works. \Cref{lem:homogeneous-holomorphic-polynomial-growth} gives the upper bound on $H_{\mathrm{aff}}$. Conversely, the affine generators have
    no common zero on $\Sigma$. Thus, for some $c_0>0$,
    \[
        \inf_{y \in \Sigma} \max_\alpha|h_\alpha(y)|\geq c_0
        \p
    \]
    If $x=\gamma_s(y)$ and $\lambda_0=\min_\alpha\lambda_\alpha>0$, choosing a nonvanishing generator at $y$ gives
    \[
        H_{\mathrm{aff}}(x) \geq c_0^2e^{\lambda_0 s},
            \qquad
        \rho^2(x) \leq r_0 e^s
        \p
    \]
    Then \cref{eq:ambient-G-f-comparison} follows from these estimates.

    By definition, we have $\sqrt{-1} \partial\bar{\partial} H_\mathrm{aff} = \omega_\mathrm{aff} \leq \omega_\cA$ and the Cauchy-Schwarz inequality shows that
    \[
        |\partial H_\mathrm{aff}|_{\omega_{\cA}}^2 \leq H_\mathrm{aff}\tr_{\omega_{\cA}}\omega_\mathrm{aff} \leq H_\mathrm{aff}
        \p
    \]
    \Cref{eq:ambient-G-derivatives} then follows from the chain rule by a direct computation.
    It remains to prove \cref{eq:ambient-G-volume-growth}. \cref{eq:ambient-reference-flow-bound} integrates to
    \[
        \gamma_s^*\omega_{\cA} \leq e^{Cs}\omega_{\mathcal A}
        \p
    \]
    Orient $\Sigma$ so that $d\sigma = (\iota_{X_0}\omega_{\cA}^n)|_\Sigma$ is positive.  Since $\Gamma(y, s) = \gamma_s(y)$,
    \[
        \Gamma^*\omega_{\cA}^n
        = ds\wedge\gamma_s^*(\iota_{X_0}\omega_{\cA}^n)
        \leq Ce^{Cs}\, ds\, d\sigma
        \p
    \]
    Moreover,  $G \leq R$ implies $s \leq C\log(2 + R)$.  Adding the compact sublevel and integrating gives
    \[
        \operatorname{Vol}_{\omega_{\cA}}\{G \leq R\}
        \leq C + C\int_\Sigma\int_0^{C\log(2 + R)}e^{Cs}\, ds\, d\sigma
        \leq C(1 + R)^D
        \p
    \]
\end{proof}

The following estimate is the key result of this section. It shows that, at least in an integral sense, the shrinker metric \(\omega\) is controlled by the ambient metric \(\omega_{\mathcal A}\) up to a polynomial factor. This estimate will allow us to prove the integrability, with respect to \(\omega\), of holomorphic vector fields commuting with \(\xi\).

\begin{proposition}
    \label{prop:weighted-reverse-ambient-trace}
    There exists $P_0 < \infty$ such that
    \begin{equation}\label{eq:weighted-reverse-ambient-trace}
        \int_X \tr_{\omega_{\cA}}\omega\,
        (1 + \rho^2)^{-P_0}\, \omega^n < \infty
        \p
    \end{equation}
\end{proposition}

\begin{proof}
    Let $S$ be as in \cref{lem:projective-denominator}. Locally we compute
    \begin{equation} \label{eq:iddbS}
        i \partial\bar{\partial} \log S = i \ell \partial\bar{\partial} \log \omega^n + i \partial\bar{\partial} \log \sum_{j=0}^M |s_j|^2 = - \ell \Ric(\omega) + \Psi^* \omega_\mathrm{FS}
        \p
    \end{equation}
    Since $\omega = \Ric(\omega) + i \partial\bar{\partial} f$ and $\omega_\cA = \frac1\ell\Psi^* \omega_\mathrm{FS} + i \partial\bar{\partial} H_\mathrm{aff}$ we see that
    \begin{equation}\label{eq:ambient-global-potential}
        u = H_{\mathrm{aff}} - \rho^2 + \frac1\ell\log S
            \qquad\text{satisfies}\qquad
        \omega_{\cA} = \omega + \sqrt{-1}\partial\bar\partial u
        \p
    \end{equation}
    By \cref{lem:projective-denominator} and \cref{eq:ambient-G-f-comparison}, there exists some $U <\infty$ such that
    \begin{equation}\label{eq:ambient-u-G-growth}
        |u| \leq C(1 + G)^U
        \p
    \end{equation}

    The volume estimate \cref{eq:ambient-G-volume-growth} implies, by a dyadic decomposition,
    \begin{equation}\label{eq:ambient-polynomial-weight-integrability}
        \int_X(1 + G)^{-q}\, \omega_{\cA}^n < \infty
        \qquad(q > D)
        \p
    \end{equation}
    Fix $Q>U+D$.  Choose a smooth nonincreasing function $\eta:[0,\infty)\to[0,1]$, equal to one on $[0,1]$ and zero on $[2,\infty)$, and set
    \[
        w_R=\eta(G/R)(1+G)^{-Q},\qquad R\geq1
        \p
    \]
    Properness of $G$ makes $w_R$ compactly supported, and $w_R\nearrow(1+G)^{-Q}$.  The chain and product rules, together with \cref{eq:ambient-G-derivatives}, give
    \begin{equation} \label{eq:ambient-weight-laplacian}
        |\Delta_{\omega_{\cA}}w_R| \leq C_Q(1+G)^{-Q-1}
    \end{equation}
    since on $\{R \leq G \leq 2R\}$, the factors $R^{-1}$ and $R^{-2}$ are controlled by $(1+G)^{-1}$ and $(1+G)^{-2}$; away from this annulus, the same estimate follows directly.

    Since $\omega = \omega_{\cA} - \sqrt{-1} \partial\bar\partial u$, compactly supported integration by parts gives
    \begin{align}
        \int_X w_R \tr_{\omega_{\cA}}\omega\, \omega_{\cA}^n
             & = n\int_X w_R\, \omega_{\cA}^n - n\int_X w_R \sqrt{-1}\partial\bar\partial u \wedge\omega_{\cA}^{n - 1}\\
             & = n\int_X w_R\, \omega_{\cA}^n - \int_X u\, \Delta_{\omega_{\cA}}w_R\, \omega_{\cA}^n \\
             &\leq \int_X\left((1+G)^{-Q}+(1+G)^{U-Q-1}\right) \omega_{\cA}^n
              < \infty
        \p
    \end{align}
    The trace is nonnegative, so monotone convergence yields
    \begin{equation}\label{eq:weighted-reverse-trace-ambient-measure}
        \int_X \tr_{\omega_{\cA}}\omega\, (1+G)^{-Q}\omega_{\cA}^n < \infty
        \p
    \end{equation}

    Finally, \cref{prop:ambient-volume-upper} gives
    \[
        \omega^n \leq C(1+\rho^2)^B \omega_{\cA}^n \leq C (1 + \rho^2)^{B D} \omega_{\cA}^n
    \]
    for some $B > 0$ and $D$ from \cref{eq:ambient-G-f-comparison}.  Thus, for $P_0=B+DQ$,
    \begin{align}
        \int_X \tr_{\omega_{\cA}}\omega\, (1 + \rho^2)^{-P_0}\omega^n
             & \leq C\int_X \tr_{\omega_{\cA}}\omega\, (1 + \rho^2)^{B - P_0}\omega_{\cA}^n\\
             & \leq C\int_X \tr_{\omega_{\cA}}\omega\, (1+G)^{-Q}\omega_{\cA}^n < \infty
        \c
    \end{align}
    proving the claim.
\end{proof}

\section{Uniform subquadratic scalar curvature bound}\label{sec-ambient comparsion II}
In this section, we establish further estimates for the shrinker metric, using the same notation as in \cref{sec:ambient-reference-comparison}. In particular, we give an analytic proof of the main result in \cite{SunZhangFano}. The rest of the paper does not rely on the results of this section.

\subsection{Subquadratic estimate for scalar curvature}

We continue with the notation fixed at the beginning of \cref{sec:ambient-reference-comparison}.  Let
    \begin{equation}\label{eq:entropy}
        \mu = \log\left(\frac{1}{(2\pi)^n n!}\int_X e^{-\rho^2}\, \omega^n\right) \leq 0.
    \end{equation}

   \begin{theorem}\label{thm:scalar}
        For every $x \in X$ with $\rho^2(x) \geq \max\{2n, e^2\}$, we have
        \begin{equation}\label{eq:scalar-bound}
            \frac{R_\omega(x)}{\rho^2(x)} \leq \frac{14n}{\rho(x)\log\rho^2(x)} + 4\sqrt{\frac{-\mu}{\log\rho^2(x)}}.
        \end{equation}
        In particular, $R_\omega = o(\rho^2)$ at infinity.
    \end{theorem}

        The proof below is purely analytic and does not use \cite{SunZhangFano}, therefore it gives an analytic proof of the compactness of $\{\nabla f = 0\}$ established there. Moreover, combining \cref{eq:scalar-bound} with the uniform entropy lower bound in \cite{Zhang2026}, and the shrinker identity $R_\omega+|\nabla^{1,0}f|^2=\rho^2$,
        it gives the proof of \cref{thm-uniform bound on critical values}.

\

In addition to \cref{eq:self-similar-identities}, we further have
    \begin{equation}\label{eq:potential-gradient}
      \xi = J\nabla^{\omega_t} F,\quad  |\partial F|_{\omega_t}^2 = F - |t|^2 R_{\omega_t}, \quad \sqrt{-1}\partial\bar\partial F = \omega_t - |t|\Ric(\omega_t).
    \end{equation}
By \cref{eq:potential-gradient}, at time $-1 - \tau$ we have
    \begin{align}
        |\xi|^2 & = 2\bigl(F - (1 + \tau)^2 R_{\omega_{-1 - \tau}}\bigr) \leq 2F, \label{eq:killing-norm}\\
        |\nabla\xi|^2 & = 2n - 4(1 + \tau)R_{\omega_{-1 - \tau}} + 2(1 + \tau)^2|\Ric|^2 \\
        & \leq 2n + 8|\Ric|^2, \qquad 0 < \tau \leq 1. \label{eq:killing-gradient}
    \end{align}

Fix $x \in X$ with $\rho^2(x) \geq \max\{2n, e^2\}$.
Let $H$ be the minimal heat kernel of $\partial_t - \Delta_{\omega_t}$, with respect to the volume form $\omega_t^n/n!$. For $t < -1$, set $\tau = -1 - t > 0$ and define $b$ by
\begin{equation}\label{eq:heat-measure}
    H(x, -1; y, -1 - \tau) = (2\pi\tau)^{-n} e^{-b(y, \tau)}
    \p
\end{equation}
Define  a measure
\begin{equation}
    d\nu_\tau = H(x, -1; \cdot, -1 - \tau)\frac{\omega_{-1 - \tau}^n}{n!}
    \p
\end{equation}
This is a probability measure by \cite[(3.3)]{LWII}.
Let
\begin{equation}\label{eq:entropy-tensor}
    E_\tau = \Ric(g_{-1 - \tau}) + \nabla^2 b - \frac{1}{\tau}g_{-1 - \tau}.
\end{equation}
In our complex normalization, the heat-kernel estimates of Li--Wang \cite{LWII} give
    \begin{align}
        \tau\int_X (\partial_w^{(x)}\log H)^2\, d\nu_\tau & \leq |w|_\omega^2, \qquad w \in T_x X, \label{eq:fisher-input}\\
        \int_X h^2\, d\nu_\tau & \leq  \left(\int_X h\, d\nu_\tau\right)^2 +2\tau\int_X |\partial h|^2\, d\nu_\tau, \label{eq:heat-poincare}\\
        \int_X R_{\omega_{-1 - \tau}}\, d\nu_\tau  \leq \frac{n}{\tau},&\quad
        \int_a^1 \int_X |\Ric|^2\, d\nu_\tau\, d\tau  \leq \frac{n}{a}, \label{eq:scalar integral bound}\\
        \int_a^1 \tau \int_X |E_\tau|^2\, d\nu_\tau\, d\tau & \leq -2\mu, \qquad 0 < a < 1. \label{eq:heat-energy}
    \end{align}
Here \cref{eq:fisher-input} follows from \cite[Theorem~3.16, (3.36)]{LWII}, and \cref{eq:heat-poincare} is \cite[Theorem~3.17]{LWII} with exponent $2$. The two estimates in \cref{eq:scalar integral bound} follow respectively from \cite[Corollary~3.20, (3.54)]{LWII} and \cite[(A.4), in the proof of Lemma~A.1]{LWII}. Finally, \cref{eq:heat-energy} follows from  \cite[(A.5)]{LWII}.

\begin{lemma}\label{lem:fisher}
    For every $0 < \tau \leq 1$,
    \begin{equation}\label{eq:killing-fisher}
        \frac{\tau}{2\rho^2(x)}\int_X (\xi b)^2\, d\nu_\tau \leq 1 - \frac{R_\omega(x)}{\rho^2(x)}.
    \end{equation}
\end{lemma}

\begin{proof}
 Since  $\xi$ is a Killing vector field, we have
    \[
        H(\exp(s\xi)x, -1; \exp(s\xi)y, -1 - \tau) = H(x, -1; y, -1 - \tau).
    \]
    Differentiation gives $\xi_x \log H = -\xi_y \log H = \xi_y b$, where the subscripts indicate the variable on which $\xi$ acts. Applying \cref{eq:fisher-input} with $w = \xi(x)$ and dividing by $2\rho^2(x)$ gives the desired estimate.
\end{proof}

 \begin{lemma}\label{lem:moments}
    For $0 < \tau \leq 1$,
    \begin{align}
        \int_X F\, d\nu_\tau & = \rho^2(x) + n\tau, \label{eq:potential-first-moment}\\
        \int_X F^2\, d\nu_\tau & \leq \bigl(\rho^2(x) + n\tau\bigr)^2 + 2\tau\bigl(\rho^2(x) + n\tau\bigr). \label{eq:potential-second-moment}
    \end{align}
    Consequently,
    \begin{align}
        \int_X |\xi|^4\, d\nu_\tau & \leq 16\rho^4(x), \label{eq:killing-fourth-moment}\\
        1 - \frac{1}{2\rho^2(x)}\int_X |\xi|^2\, d\nu_\tau & \leq \frac{n}{\rho^2(x)}(2 + \tau^{-1}). \label{eq:killing-moments}
    \end{align}
\end{lemma}

\begin{proof}
    Choose a smooth nondecreasing concave function $\sigma$ on $[0, \infty)$ which equals the identity on $[0, 1]$ and is constant on $[2, \infty)$. Put $\sigma_L(s) = L\sigma(s/L)$. Heat duality and \cref{eq:self-similar-identities} give
    \begin{equation}\label{eq:truncated-moment}
        \int_X \sigma_L(F)\, d\nu_\tau - \sigma_L(\rho^2(x)) 
     = \int_0^\tau \frac{d}{d \tau} \int_X \sigma_L(F) \d \nu_\tau \d \tau
            = \int_0^\tau\int_X \left[n\sigma_L'(F) + \sigma_L''(F)|\partial F|^2\right]d\nu_s \, ds.
    \end{equation}
    For fixed $L$, subtracting the constant value at infinity makes this a compactly supported calculation.   The integrand on the right is uniformly bounded by \cref{eq:potential-gradient}, and converges pointwise to $n$. Dominated convergence on the right and monotone convergence on the left prove \cref{eq:potential-first-moment}.

    Apply \cref{eq:heat-poincare} to $\sigma_L(F)$. Since $|\partial\sigma_L(F)|^2 \leq F$, monotone convergence gives \cref{eq:potential-second-moment}. Thus \cref{eq:killing-norm} implies
    \[
        \frac{1}{4\rho^4(x)}\int_X |\xi|^4\, d\nu_\tau \leq \left(1 + \frac{n\tau}{\rho^2(x)}\right)^2 + \frac{2\tau}{\rho^2(x)}\left(1 + \frac{n\tau}{\rho^2(x)}\right) \leq 4.
    \]
    Finally,
    \[
        \frac{1}{2\rho^2(x)}\int_X |\xi|^2\, d\nu_\tau = 1 + \frac{n\tau}{\rho^2(x)} - \frac{(1 + \tau)^2}{\rho^2(x)}\int_X R_{\omega_{-1 - \tau}}\, d\nu_\tau.
    \]
    The scalar-curvature integral estimate \cref{eq:scalar integral bound} gives \cref{eq:killing-moments}.
\end{proof}
\begin{lemma}\label{lem:weighted-killing}
    For every $\tau \in (0, 1]$,
    \begin{equation}\label{eq:weighted-killing}
        \int_X \bigl((\xi b)^2 + |\nabla\xi|^2\bigr)\, d\nu_\tau = \frac{1}{\tau}\int_X |\xi|^2\, d\nu_\tau + \int_X E_\tau(\xi, \xi)\, d\nu_\tau.
    \end{equation}
\end{lemma}
\begin{proof}
  At a fixed time, the Killing identities give
    \begin{align}
        \operatorname{div}\xi & = 0, \label{eq:killing-divergence}\\
        \nabla_\xi\xi & = -\tfrac12\nabla|\xi|^2, \label{eq:killing-acceleration}\\
        \Delta_{\omega_{-1 - \tau}}|\xi|^2 & = |\nabla\xi|^2 - \Ric(g_{-1 - \tau})(\xi, \xi). \label{eq:killing-bochner}
    \end{align}
    Choose a smooth cutoff $\eta_L = \eta(F/L)$, equal to one where $F \leq L$ and zero where $F \geq 2L$. Since $\xi F = 0$, we have $\xi\eta_L = 0$. Apply the divergence theorem to the compactly supported vector field $e^{-b}\eta_L^2(\xi b)\xi$. Using $\operatorname{div}\xi = 0$, $\xi\eta_L = 0$, and \cref{eq:heat-measure}, we obtain
    \[
        0 = \int_X \eta_L^2\bigl(\xi(\xi b) - (\xi b)^2\bigr)\, d\nu_\tau.
    \]
   Since $\xi(\xi b) = \nabla^2 b(\xi, \xi) + db(\nabla_\xi\xi)$, \cref{eq:killing-acceleration} yields
    \[
        \int_X \eta_L^2(\xi b)^2\, d\nu_\tau = \int_X \eta_L^2\left(\nabla^2 b(\xi, \xi) - \operatorname{Re}\langle\partial b, \partial|\xi|^2\rangle\right)d\nu_\tau.
    \]
    A second integration by parts, using the density $e^{-b}$, gives
    \[
        \int_X \eta_L^2\Delta_{\omega_{-1 - \tau}}|\xi|^2\, d\nu_\tau = \operatorname{Re}\int_X \left(\eta_L^2\langle\partial b, \partial|\xi|^2\rangle - \langle\partial(\eta_L^2), \partial|\xi|^2\rangle\right)d\nu_\tau.
    \]
    Combining these identities with \cref{eq:killing-bochner,eq:entropy-tensor}, we obtain
    \begin{align}
        \int_X \eta_L^2\bigl((\xi b)^2 + |\nabla\xi|^2\bigr)\, d\nu_\tau
        & = \int_X \eta_L^2\left(\frac{|\xi|^2}{\tau} + E_\tau(\xi, \xi)\right)d\nu_\tau \\
        &\quad - \operatorname{Re}\int_X \langle\partial(\eta_L^2), \partial|\xi|^2\rangle\, d\nu_\tau. \label{eq:killing-cutoff}
    \end{align}
    The preceding cutoff identities hold for every $\tau \in (0, 1]$. By \cref{eq:scalar integral bound,eq:heat-energy} and Fubini's theorem, for almost every $\tau\in (0,1]$ both $\int_X |\Ric|^2\, d\nu_\tau$ and $\int_X |E_\tau|^2\, d\nu_\tau$ are finite. Fix such a $\tau$. Then $\int_X |\nabla\xi|^2\, d\nu_\tau < \infty$ by \cref{eq:killing-gradient}, and $\int_X |E_\tau||\xi|^2\, d\nu_\tau < \infty$ by \cref{lem:moments}. On the cutoff annulus $\{L \leq F \leq 2L\}$,
    \[
        |\partial\eta_L|^2|\xi|^2 \leq \frac{CF^2}{L^2}\mathbf 1_{\{L \leq F \leq 2L\}} \leq C\mathbf 1_{\{F \geq L\}}.
    \]
    The absolute value of the boundary term in \cref{eq:killing-cutoff} is bounded by
    \[
        C\left(\int_X |\nabla\xi|^2\, d\nu_\tau\right)^{1/2}\left(\int_X |\partial\eta_L|^2|\xi|^2\, d\nu_\tau\right)^{1/2} \longrightarrow 0.
    \]
    All other terms on the right are integrable by the preceding bounds. Passing to the limit proves \cref{eq:weighted-killing}.
\end{proof}

\begin{proof}[Proof of \cref{thm:scalar}]
    By \cref{lem:fisher} and \cref{lem:weighted-killing}, for  every $\tau \in (0, 1]$,
    \begin{align}
        \frac{R_\omega(x)}{\rho^2(x)} & \leq 1 - \frac{1}{2\rho^2(x)}\int_X |\xi|^2\, d\nu_\tau + \frac{\tau}{2\rho^2(x)}\int_X |\nabla\xi|^2\, d\nu_\tau \\
        &\quad - \frac{\tau}{2\rho^2(x)}\int_X E_\tau(\xi, \xi)\, d\nu_\tau. \label{eq:scalar-defect}
    \end{align}
    Fix $0 < a < 1$. Integrating the first two terms on the right against $d\tau/\tau$ and using \cref{eq:killing-moments,eq:killing-gradient,eq:scalar integral bound} gives
    \begin{align}
        \int_a^1\left(1 - \frac{1}{2\rho^2(x)}\int_X |\xi|^2\, d\nu_\tau\right)\frac{d\tau}{\tau} & \leq \frac{n}{\rho^2(x)}\bigl(2\log(1/a) + a^{-1} - 1\bigr), \\
        \frac{1}{2\rho^2(x)}\int_a^1\int_X |\nabla\xi|^2\, d\nu_\tau\, d\tau & \leq \frac{n(1 - a)}{\rho^2(x)} + \frac{4n}{a\rho^2(x)}.
    \end{align}
    For the last term, first Cauchy--Schwarz and then \cref{eq:heat-energy,eq:killing-fourth-moment} give
    \begin{align}
        \frac{1}{2\rho^2(x)}\left|\int_a^1\int_X E_\tau(\xi, \xi)\, d\nu_\tau\, d\tau\right|
        & \leq \frac{1}{2\rho^2(x)}\left(\int_a^1\tau\int_X |E_\tau|^2\, d\nu_\tau\, d\tau\right)^{1/2} \\
        &\quad\times\left(\int_a^1\tau^{-1}\int_X |\xi|^4\, d\nu_\tau\, d\tau\right)^{1/2} \\
        & \leq 2\sqrt{-2\mu\log(1/a)}.
    \end{align}
    Since $\log(1/a) \leq a^{-1}$, integration of \cref{eq:scalar-defect} yields
    \begin{equation}\label{eq:integrated-scalar}
        \frac{R_\omega(x)}{\rho^2(x)}\log(1/a) \leq \frac{7n}{a\rho^2(x)} + 2\sqrt{-2\mu\log(1/a)}.
    \end{equation}
    Taking $a = \rho(x)^{-1}$ proves \cref{eq:scalar-bound}.
\end{proof}

\subsection{Rough two sided bound of the shrinker metric}

In this section, we compare the shrinker metrics with ambient K\"ahler metrics without imposing any curvature assumptions.

Let $ S = \sum_{j = 0}^M|s_j|_{h_\omega^\ell}^2 $ as in \cref{lem:projective-denominator}, which is a smooth positive function on $X$.

\begin{proposition}
    \label{prop:projective-trace-exponential}
    There is $C < \infty$ such that
    \begin{equation}\label{eq:projective-trace-exponential}
        \frac1\ell\tr_\omega\Psi^*\omega_{\rm FS}
        \leq (1+\rho)^C S^{-C}
        \p
    \end{equation}
\end{proposition}

\begin{proof}
    Set
    \[
        U = \frac1\ell\tr_\omega \Psi^*\omega_{\rm FS}
            \qquad
        \chi = -\frac1\ell\log S
        \p
    \]
    By the Chern--Lu inequality and the shrinker equation, there is $C_0 < \infty$, depending only on the fixed projective system, such that
    \begin{equation}
        \Delta_f \log(U + 1) \geq -C_0 U - C_0,
        \p
    \end{equation}
    Furthermore contracting \cref{eq:iddbS} and the homogeneity of $s_j$ gives
    \[
        \Delta_f \chi = n - U + O(1)
    \]
    Hence picking $a > C_0 + 1$ and $B \gg a$, we can define
    \[
        H = (U + 1)e^{-a\chi}(1+\rho^2)^{-B}
        \p
    \]
    which satisfies $H \leq C(U + 1)$ by \cref{lem:projective-denominator}. Since $\Delta_f\log(1 + \rho^2)$ is bounded, we have
    \begin{equation}\label{eq:projective-barrier-inequality}
        \Delta_f \log H \geq (a - C_0)U - C \geq U - C
        \p
    \end{equation}
    The bound $H \leq C(U + 1)$ implies $U \geq cH - 1$.  Consequently,
    \begin{equation}
        \Delta_f H = H\bigl(\Delta_f \log H + |\partial\log H|_\omega^2\bigr)
            \geq H \left( U - C \right)
            \geq cH^2 - CH
        \p
    \end{equation}

    To justify the maximum principle on noncompact $X$, choose a nonincreasing cutoff $\eta$ supported in $[0, 2]$, equal to one on $[0, 1]$, and satisfying $(\eta')^2 \leq C\eta$ and $|\eta''| \leq C$.  For $R \geq 2n + 1$, let $\eta_R = \eta(\rho^2/R)$.  Using $\Delta_f \rho^2 = n - \rho^2$ and $|\partial\rho^2|_\omega^2 = \rho^2 - R_\omega \leq \rho^2$, we see that
    \[
        \frac{|\partial\eta_R|_\omega^2}{\eta_R} \leq \frac CR,
            \qquad
        \Delta_f \eta_R \geq -\frac CR
            \quad\text{on }\{\eta_R > 0\}
        \p
    \]
    At a maximum of $\eta_R H$, \cref{eq:projective-barrier-inequality} and $\partial H = -H\partial\eta_R/\eta_R$ give
    \[
        0 \geq c\eta_R H^2 - C\eta_R H - \frac CRH
        - 2H\frac{|\partial\eta_R|_\omega^2}{\eta_R}
        \p
    \]
    After division by $H$ and use of the cutoff estimate, this gives
    \[
        c\eta_R H \leq C\eta_R + \frac CR
        \p
    \]
    Thus $\sup_X \eta_R H \leq C$, uniformly in $R$, and exhaustion gives $H \leq C$.  Therefore
    \[
        U + 1 \leq C (1+\rho^2)^B e^{a\chi }
        \leq (1+\rho^2)^B S^{-C},
    \]
    which proves \cref{eq:projective-trace-exponential}.
\end{proof}

\Cref{prop:ambient-lower-bound} and
\cref{eq:projective-trace-exponential} give the lower metric bound in
\cref{eq:ambient-metric-comparison}.  Together with \cref{prop:ambient-volume-upper}, pointwise linear algebra gives the reverse
bound.

\begin{corollary}
    \label{cor:two-sided-ambient-metric-comparison}
    There is $C < \infty$ such that
    \begin{equation}\label{eq:ambient-metric-comparison}
         (1+\rho^2)^{-C} S^{C}\omega_{\cA}
        \leq \omega \leq 
        (1+\rho^2)^C S^{-C}\omega_{\cA}
        \p
    \end{equation}
\end{corollary}

\section{Noncompact Matsushima theorem}
\label{sec:weighted-matsushima}

The ambient comparison gives the weighted integrability needed for the
noncompact Matsushima theorem. The resulting reductive structure then yields
uniqueness of the soliton field.

Throughout this section, write
\[
    d\mu_f = e^{-f}\omega^n
    \p
\]
For a vector field $Y$, set
\[
    \div_f Y
    = \div_{\omega^n}Y - Yf
    \p
\]
Thus $\div_f$ is the divergence with respect to $d\mu_f$, and compactly supported weighted integration by parts reads
\begin{equation}\label{eq:section-three-weighted-divergence}
    \int_X \psi\, \div_f Y\, d\mu_f
    = -\int_X Y(\psi)\, d\mu_f
    \p
\end{equation}
Fix a nonincreasing $\chi\in C_c^\infty([0,\infty))$, equal to one on
$[0,1]$ and zero on $[2,\infty)$, and set
$\chi_R=\chi(\rho^2/R)$.  Then
\begin{equation}\label{eq:section-three-cutoff-gradient}
    |\nabla\chi_R|_\omega \leq CR^{-1/2}
    \quad\text{on }\{R \leq \rho^2 \leq 2R\}
    \p
\end{equation}

\subsection{Weighted integrability of the holomorphic centralizer}

Let $\xi = J\nabla f$ be the soliton vector field.  Denote by $\mathfrak h_\xi$ the Lie algebra of holomorphic vector fields commuting with $\xi$.  For
$Y\in\mathfrak h_\xi$, let
\begin{equation}\label{eq:explicit-soliton-potential}
    \theta_Y
   : = Y^{1,0}f - \div_{\omega^n}Y^{1,0}=-\div_f Y^{1,0}
    \p
\end{equation}
Holomorphicity gives
\[
    \iota_{Y^{1,0}} \Ric(\omega) = -\sqrt{-1}\, \bar\partial (\div_{\omega^n}Y^{1,0}),
        \qquad
    \iota_{Y^{1,0}} \sqrt{-1}\partial\bar\partial f = \sqrt{-1}\, \bar\partial(Y^{1,0}f)
    \p
\]
Therefore the shrinker equation implies that $\theta_Y$ is a Hamiltonian for $Y^{1,0}$, which can be expressed in any of the following equivalent equations:
\begin{align} \label{eq:HamilY}
    \iota_{Y^{1,0}} \omega &= \sqrt{-1} \bar\partial \theta_Y \\
     \nabla^{1, 0} \theta_Y &= Y^{1,0} \\
            \iota_Y \omega &= \sqrt{-1}\bigl(\bar\partial\theta_Y - \partial\overline{\theta_Y}\bigr)
    \p
\end{align}
Moreover, $\xi(f)=0$, $\cL_\xi\omega^n=0$, and
$[\xi,Y^{1,0}]=0$, imply
\[
    \xi(\div_{\omega^n}Y^{1,0})
    = \div_{\omega^n}[\xi, Y^{1,0}] = 0,
\]
and hence $\xi(\theta_Y) = 0$.

\begin{lemma}
    \label{lem:centralizer-weighted-integrability}
    For every $Y \in \mathfrak h_\xi$,
    \[
        Y \in L^2(\omega, \d \mu_f),
        \qquad
        \theta_Y \in W^{1, 2}( \d \mu_f)
        \p
    \]
\end{lemma}

\begin{proof}
    Since $[Y, \xi] = 0$, the derivation induced by $Y$ preserves the eigenspaces of $\xi$ in $\cO(X)$.  In particular, each $Y h_\alpha$ is a finite linear combination of homogeneous functions of the same weight.  Thus \cref{lem:homogeneous-holomorphic-polynomial-growth} gives
    \begin{equation} \label{eq:YaffBound}
        |Yh_\alpha| \leq C(1+\rho^2)^{\lambda_\alpha}
        \p
    \end{equation}
    Since $\cL_Y s_j$ has the same weight as $s_j$, \zcref{rmk:normalLowWt} shows that
    \[
        \cL_Y s_j = \sum_i a_{ij} s_i
    \]
    with $a_{ij} \in \bbC$ and thus $|d \Psi(Y)|_{\omega_\mathrm{FS}} \leq C$. Together with \cref{eq:YaffBound} we obtain
    \[
        |Y|_{\omega_{\cA}} \leq C(1+\rho^2)^A
    \]
    for some $A > 0$.
    Pointwise linear algebra gives
    \[
        |Y|_\omega^2 \leq \tr_{\omega_{\cA}}\omega\, |Y|_{\omega_{\cA}}^2
        \p
    \]
    Since $(1+\rho^2)^{2A}e^{-f} \leq C_Ae^{-f/2}$ outside a compact set, \cref{prop:weighted-reverse-ambient-trace} gives
    \begin{align}
        \int_X|Y|_\omega^2 \d \mu_f &\leq C\int_X \tr_{\omega_{\cA}}\omega\, (1+\rho^2)^{2A}e^{-f}\omega^n\\
                          &\leq C_A \int_X \tr_{\omega_{\cA}}\omega\, e^{-f/2}\omega^n < \infty
        \c
    \end{align}
    showing that $Y \in L^2(\omega, \d \mu_f)$.

    Applying \cref{eq:explicit-soliton-potential} and \cref{eq:section-three-weighted-divergence} with $\psi = \chi_R^2\overline{\theta_Y}$ we obtain
    \[
        \int_X \chi_R^2|\theta_Y|^2 \d \mu_f = \int_X Y^{1, 0}(\chi_R^2\overline{\theta_Y}) \d \mu_f
    \]
    Using $Y^{1, 0}(\overline{\theta_Y}) = |Y^{1, 0}|_\omega^2 = \frac12 |Y|_\omega^2$ then gives
    \begin{align}
        \int_X \chi_R^2|\theta_Y|^2 \d \mu_f
             &\leq \frac12\int_X \chi_R^2|Y|_\omega^2 \d \mu_f + 2\int_X \chi_R|\theta_Y|\, |Y^{1, 0}\chi_R| \d \mu_f\\
             &\leq \frac12\int_X \chi_R^2|\theta_Y|^2 \d \mu_f + C\int_X(\chi_R^2 + |\nabla\chi_R|_\omega^2)|Y|_\omega^2 \d \mu_f
         \p
    \end{align}
    Since $|\nabla\chi_R|_\omega^2 \leq C/R$ on $\{R \leq \rho^2 \leq 2R\}$, the second integrand is bounded by $2 C |Y|_\omega^2$. Absorbing the first integral on the RHS into the LHS and letting $R \to \infty$ we obtain
    \begin{equation}\label{eq:weighted-potential-L2}
        \int_X|\theta_Y|^2 \d \mu_f
        \leq C\int_X|Y|_\omega^2 \d \mu_f < \infty
        \p
    \end{equation}
    Crucially, the $\xi$-invariance of $\theta_Y$ gives $\frac12 (\nabla f) \theta_Y = (\nabla^{1,0} f) \theta_Y$ and thus $\frac12 (\nabla f) \theta_Y = Y^{1,0} f$ by \cref{eq:HamilY}.
    Therefore by definition of $\div_f$, we obtain 
    \[
        \Delta_f \theta_Y = \Delta_\omega \theta_Y - \frac12 (\nabla f) \theta_Y = \div_f Y^{1,0} = -\theta_Y
        \p
    \]
    Testing this equation against $\chi_R^2\overline{\theta_Y}$ gives 
    \begin{align}
        \int_X \chi_R^2 |\theta_Y|^2 \d \mu_f &= - \int_X \chi_R^2 \overline{\theta_Y} \Delta_f \theta_Y \d \mu_f \\
                                    &= \frac12 \int_X \langle \nabla (\chi_R^2 \theta_Y), \nabla \theta_Y \rangle_\omega \d \mu_f \\
                                    &= \frac12 \int_X \chi_R^2 |\nabla \theta_Y|_\omega^2 \d \mu_f + \int_X \chi_R \overline{\theta_Y} \langle \nabla \chi_R, \nabla \theta_Y \rangle_\omega \d \mu_f
        \p
    \end{align}
    After using Cauchy-Schwarz and absorbing terms we obtain
    \[
        \int_X \chi_R^2|\nabla\theta_Y|_\omega^2 \d \mu_f \leq C\int_X(\chi_R^2 + |\nabla\chi_R|_\omega^2) |\theta_Y|^2 \d \mu_f
        \p
    \]
    Letting $R \to \infty$ proves $\nabla\theta_Y \in L^2( \d \mu_f)$, and hence $\theta_Y \in W^{1, 2}( \d \mu_f)$.
\end{proof}

\subsection{The Matsushima theorem}
Let $\Aut_0(X,\xi)$ denote the identity component of the group of automorphisms of $X$ that commute with $\xi$.
Let $\mathfrak k_\xi$ be the real Lie algebra of real holomorphic Killing
fields of $(X,\omega,J)$ that commute with $\xi$.

\begin{proposition}[$=$\cref{thm:introduction-subcritical-matsushima}]
    \label{prop:weighted-matsushima}
    We have
    \begin{equation}\label{eq--cartan decomposition}
         \mathfrak h_\xi = \mathfrak k_\xi\oplus J\mathfrak k_\xi\simeq\mathfrak k_\xi^{\bbC}
        \p
    \end{equation}
    In particular, any compact torus in $\Aut_0(X,\xi)$ is conjugate, by an element of $\Aut_0(X,\xi)$, to a compact torus acting by holomorphic isometries.
\end{proposition}

\begin{proof}
    By \cref{lem:centralizer-weighted-integrability}, every $Y\in\mathfrak h_\xi$ has a potential $\theta_Y\in W^{1,2}(d\mu_f)$ satisfying
    \[
        Y^{1,0} = \nabla^{1, 0}\theta_Y,
        \qquad
        \bar\partial \nabla^{1, 0}\theta_Y = 0
        \p
    \]
    It also satisfies
    \[
        \xi(\theta_Y) = 0,
        \qquad
        \Delta_f \theta_Y + \theta_Y = 0
        \p
    \]

    The soliton Matsushima calculation
    \cite[Section~2, especially equation~(2.7)]{TianZhuUniqueness} applies
    once its integrations by parts are justified by cutoff, as in
    \cite[Claim~5.6]{CDS}. Let $\bar\partial_f^*$ denote the formal adjoint of $\bar\partial$ on $T^{1, 0}X$ with respect to $\omega$ and $d\mu_f$. Since $\xi(\overline{\theta_Y}) = 0$ and $\Delta_f$ is real, the K\"ahler commutation formula and the shrinker equation give
    \[
        \bar\partial_f^*\bar\partial\nabla^{1, 0}\overline{\theta_Y}
        = -\nabla^{1, 0}(\Delta_f\overline{\theta_Y} + \overline{\theta_Y}) = 0
        \p
    \]
    Multiplying by $\chi_R^2\nabla^{1,0}\overline{\theta_Y}$, integrating against $d\mu_f$, and applying the Cauchy--Schwarz inequality yield
    \begin{align}
        \int_X \chi_R^2|\bar\partial\nabla^{1, 0}\overline{\theta_Y}|_\omega^2d\mu_f
        & \leq C\int_X |\bar\partial\chi_R|_\omega^2|\nabla^{1, 0}\overline{\theta_Y}|_\omega^2d\mu_f\\
        & \leq \frac{C}{R}\int_{\{R \leq \rho^2 \leq 2R\}}|\nabla\theta_Y|_\omega^2d\mu_f
        \p
    \end{align}
    Using \cref{lem:centralizer-weighted-integrability} and  letting $R \to \infty$, we obtain that $\nabla^{1, 0}\overline{\theta_Y}$ is holomorphic.

    Let $a = \Re \theta_Y$ and $b = \Im \theta_Y$. Then
    \[
        Y = \nabla a + J \nabla b
        \p
    \]
    Since both $\nabla^{1, 0}\theta_Y$ and $\nabla^{1, 0}\overline{\theta_Y}$ are holomorphic, the fields \( \nabla^{1, 0}a, \nabla^{1, 0}b \) are holomorphic.  For a real function $a$, holomorphicity of $\nabla^{1,0}a$ makes $J\nabla a$ real holomorphic, and since it is always symplectic, we show that it is Killing.  The same applies to $b$.  Since $\xi(\theta_Y) = 0$, we have
    \[
        J\nabla a,  J\nabla b\in \mathfrak k_\xi
        \p
    \]
    Since $Y \in \fh_\xi$ was arbitrary, we have proven
    \[
        \mathfrak h_\xi= \mathfrak k_\xi+J\mathfrak k_\xi
        \p
    \]
    The argument in \cite[page~314]{CDS}, implies that
    $\mathfrak k_\xi \cap J\mathfrak  k_\xi= 0.$ Therefore we get the desired direct sum.

    In particular, $\mathfrak h_\xi$ is finite-dimensional. It then follows from
    \cite[Chapter~I, Theorem~3.1]{kobayashi2012transformation} that
    $\Aut_0(X,\xi)$ is a Lie group. Let $K_\xi$ be the identity component of the group of holomorphic
    isometries of $(X,\omega,J)$ that commute with $\xi$. Then
    $\Lie(K_\xi)=\mathfrak k_\xi$. Furthermore $K_\xi$ is compact since it preserves the compact set $\{\xi = 0\}$. Note that the decomposition in \cref{eq--cartan decomposition} is a
    Cartan decomposition, and hence $K_\xi$ is a maximal compact subgroup of
    $\Aut_0(X,\xi)$ \cite{helgason1979differential}; see also \cite[Corollary~5.11]{CDS} for an argument
    specialized to this setting. The desired conjugacy then follows from the
    Cartan--Iwasawa--Malcev theorem \cite{iwasawa1949} and the conjugacy theorem for maximal
    tori in compact connected Lie groups.
\end{proof}

\begin{remark}
    As first observed in \cite{CDSIII}; see also \cite[Corollary~4.9]{EAC}, the uniqueness of K\"ahler--Ricci shrinkers for a fixed soliton vector field, established in \cref{thm:fixed-soliton-field-metric-uniqueness}, implies Matsushima's theorem. This provides an alternative proof of \cref{thm:introduction-subcritical-matsushima}. Note that the alternative proof also establishes that $\Isom(X, \omega, \xi) \inj \Aut(X, \xi)$ is bijective on connected components.
\end{remark}

\subsection{Uniqueness of the soliton vector field}

We now combine the noncompact Matsushima theorem with the weighted-volume functional $\bbW$ of \cite{SunZhangFano} to prove the uniqueness of the soliton vector field up to a biholomorphism. Note that our proof only uses the Matsushima theorem, but none of the more analytic results in this or the previous section.

\begin{proposition} \label{prop:soliton-critical-after-matsushima}
    Let $(X,\omega,f)$ be as above, and let $\bbT$ be a compact torus in $\Aut(X,\xi)$ that contains $\xi$. Then $\xi$ is a critical point of $\bbW$ on the Reeb cone of $\bbT$.
\end{proposition}

\begin{proof}
    Since $\bbT$ is compact and contains $\xi$ in its Lie algebra,  \cref{prop:weighted-matsushima} gives an element $g\in \Aut(X,\xi)$  such that $g\bbT g^{-1}$ acts by isometries of $(X,\omega)$.  Because $g$ fixes $\xi$ and $\bbW$ is conjugation-equivariant, we may assume that $\omega$ is $\bbT$-invariant in what follows.

    Let $\eta \in \Lie(\bbT)$. Using \cite[Lemma~3.3,Lemma~5.7]{SunZhangFano}, it is direct to show 
    \begin{equation}\label{eq:killing-weighted-volume-first-variation}
        D\bbW|_\xi(\eta) = -c_n \int_X u_\eta\, d\mu_f,
    \end{equation}
    where $c_n > 0$ is dimensional and the canonical Hamiltonian $u_\eta$ is normalized by
    \begin{equation}\label{eq:killing-canonical-hamiltonian}
        \cL_{J\eta}d\mu_f = 2u_\eta d\mu_f
        \p
    \end{equation}
    Note that $u_\eta \in L^1(\d \mu_f)$ by \cref{lem:centralizer-weighted-integrability}. Note that for vector fields in a compact torus of isometries there is an easier way to see this: We have $u_\eta = \langle \mu, \eta \rangle$, and on the image of the moment map we have $| x | \leq C (\langle \xi, x \rangle + C)$, so $|u_\eta| \leq C (f + C)$.

    Let $\Phi^\eta_t$ denote the flow of $J \eta$ on $X$. Now we use that pullbacks preserve the integrals of volume forms, the Fubini-Tonelli theorem and the fundamental theorem of calculus to see
    \[
        \int_X u_\eta \d \mu_f = \int_0^1 \left( \int_X (\Phi^\eta_s)^* (u_\eta \d \mu_f) \right) \d s = \int_X  \int_0^1 (\Phi^\eta_s)^* (u_\eta \d \mu_f) \d s  = \frac12 \int_X  (\Phi^\eta_1)^* \d \mu_f - \d \mu_f = 0
        \p
    \]
\end{proof}

\begin{theorem} \label{thm:introduction-unique-soliton-vector-field}
    Let $(\omega_i, f_i, \xi_i = J\nabla^{\omega_i}f_i)$, $i = 0, 1$, be
    complete noncompact K\"ahler--Ricci shrinkers on $X$. Then there exists a
    biholomorphism $\psi \in \Aut(X)$ such that
    \[
        \psi_* \xi_0 = \xi_1
        \p
    \]
    If $\xi_0, \xi_1 \in \Lie(\bbT_{\max})$ for the same maximal compact torus $\bbT_{\max}$, then $\xi_0= \xi_1$.
\end{theorem}

\begin{proof}
    Choose maximal compact tori $\bbT_i$ containing $\xi_i$. By the conjugacy theorem for maximal tori of a polarized Fano fibration \cite[Proposition~5.17]{EAC}, an automorphism $\psi$ conjugates $\bbT_0$ to $\bbT_1$. Thus $\psi_*\xi_0$ and $\xi_1$ lie in $\Lie(\bbT_1)$, and we are reduced to showing that two soliton vector fields in the same maximal compact torus must agree.

    \Cref{prop:soliton-critical-after-matsushima} shows that both vectors are critical points of the weighted volume functional $\bbW_X$. By \cite[Lemma~5.7]{SunZhangFano}, $\bbW_X$ is strictly convex on the Reeb cone inside $\bbT$. Hence
    \[
        \psi_* \xi_0 = \xi_1
        \p
    \]
\end{proof}

\section{Pluripotential theory on polarized Fano fibrations}%
\label{sec:pluripotential_theory_on_polarized_fano_fibrations}

Let $\pi : X \to Y$ be a smooth polarized Fano fibration equipped with a torus $\bbT$ and a Reeb field $\xi \in \Lie \bbT$, and fix a homogeneous ambient system as in \cref{sec:ambient-reference-comparison}. No shrinker is assumed to exist in this section. We develop the pluripotential theory of Hermitian metrics on $-K_X$ which are quadratic with respect to $-J\xi$ using only this polarized Fano-fibration data.

\subsection{Set-up and background metrics}

Set $X_0 = -\frac{1}{2}J\xi$. The action $\bbT \acts X$ has a canonical lift to $K_X^{-1}$ and moreover extends to an action of the complexified torus $\bbT_\bbC \cong (\bbC^*)^k$. In particular,
\begin{equation}\label{eq:complex-flow}
    \gamma_{s + \sqrt{-1}\sigma}
    := \exp(sX_0)\circ\exp(\sigma JX_0)
    \p
\end{equation}
is defined for all $s, \sigma \in \bbR$, giving us a $\bbC$-action on $X$ and also $K_X^{-1}$.

For every smooth $\bbT$-invariant Hermitian metric $\psi$ on $K_X^{-1}$, define its Hamiltonian by
\begin{equation}\label{eq:general-hamiltonian}
    h^\psi := -e^\psi\cL_{X_0}(e^{-\psi})
    \p
\end{equation}
Thus $\frac{d}{d s}\gamma_s^*\psi = \gamma_s^*h^\psi$ and $\iota_\xi\sqrt{-1}\partial\bar\partial\psi = -d h^\psi$.

Let $E = \pi^{-1}(0)$ be the compact central fiber.
As in \cite[\S 4.2]{CE} we can choose a $\bbT$-invariant PSh function $r^2$ on $\bbC^N$, which is smooth on $\bbC^N \setminus \{0\}$ and satisfies $(-J \xi) r^2 = 2 r^2$, so that
\[
    \omega_C = \frac{\sqrt{-1}}{2}\partial\bar\partial r^2
\]
is a cone metric on $(\bbC^N, \xi)$
We use the same notation for its pullback to $X \setminus E$ and require
\begin{equation}\label{eq:cone-potential-homogeneity}
    X_0 r^2 = r^2
    \p
\end{equation}
We use the same notation for its continuous extension by zero over $E$. It need not be smooth along $E$, and no argument below requires this. In particular,
\begin{equation}\label{eq:cone-potential-flow}
    r^2(\gamma_s x) = e^s r^2(x)
    \p
\end{equation}
The compact hypersurface
\begin{equation}\label{eq:transversal-choice}
    \Sigma_0 := \{r = 1\}
\end{equation}
is transverse to $X_0$, and every point in $\{r \geq 1\}$ has a unique expression $x = \gamma_s y$ with $y \in \Sigma_0$ and $s \geq 0$.

Following \cite[Lemma~4.5 and Definition~4.6]{CE}, choose a smooth $\bbT$-invariant strictly psh function $h_{\mathrm{AC}}$ on $\bbC^N$ such that
\begin{equation} \label{eq:hAC}
    h_{\mathrm{AC}} = \frac{r^2}{2}
\end{equation}
outside a compact set. Recall that one obtains $h_{\mathrm{AC}}$ by choosing a suitable smooth convex increasing function $V$ and a cutoff $\zeta$ supported near the origin, and setting
\[
    h_{\mathrm{AC}} = \frac{V(r^2)}{2} + \varepsilon \zeta(r)|z|^2,
\]
where the choices are made so that $\omega_{\mathrm{AC}} := \sqrt{-1}\partial\bar\partial h_{\mathrm{AC}}$ is K\"ahler and agrees with $\omega_C$ outside a compact set. On $\bbC^N \times \bbP^M$, set
\[
    \omega_{\mathrm{amb}} = \omega_{\mathrm{AC}} + \frac{1}{\ell}\omega_{\mathrm{FS}}
    \p
\]
Its restriction to $X$ is the curvature of the smooth $\bbT$-invariant Hermitian metric $\psi_{\mathrm{amb}}$ on $K_X^{-1}$ defined in additive notation by
\begin{equation}\label{eq:defpsiamb}
    \psi_{\mathrm{amb}} = h_{\mathrm{AC}} + \frac{1}{\ell}h_{\mathrm{FS}}
    \p
\end{equation}
Both the ambient metric and its restriction to $X$ are complete. Moreover,
\begin{equation}\label{eq:section-four-ambient-flow-bound}
    -C \omega_{\mathrm{amb}}
    \leq \cL_{X_0} \omega_{\mathrm{amb}}
    \leq C \omega_{\mathrm{amb}}
    \p
\end{equation}
Indeed, the cone summand is homogeneous outside a compact set, while the estimate for the Fubini--Study summand follows from the finite-dimensional $\bbT_\bbC$-action; the estimate on the compact remainder follows from smoothness.

\subsection{Quadratic growth classes}

\begin{definition}[Metrics with quadratic growth]
    \label{def:quadratic-spaces}
    Let $\cH$ be the space of smooth Hermitian metrics $\psi$ on $K_X^{-1}$ such that $\sqrt{-1}\partial\bar\partial\psi > 0$, and let $\overline{\mathcal H}$ be the space of Hermitian metrics with locally bounded psh weights.  For $0 < m \leq M$, let $\overline{\mathcal H}_2(m, M)$ consist of those $\psi \in \overline{\mathcal H}$ for which
    \begin{equation}\label{eq:H2-bounds}
        m e^s - m^{-1}
        \leq (\gamma_s^*\psi - \psi)|_{\Sigma_0}
        \leq M e^s
        \qquad(s \geq 1)
        \p
    \end{equation}
    Set
    \[
        \cH_2(m, M)
        = \cH\cap\overline{\mathcal H}_2(m, M),
        \qquad
        \cH_2 = \bigcup_{m, M}\mathcal H_2(m, M),
        \qquad
        \overline{\cH}_2
        = \bigcup_{m, M}\overline{\cH}_2(m, M)
        \p
    \]
    A superscript $\bbT$ denotes the $\bbT$-invariant part of any of these spaces.  We call a family in $\overline{\cH}_2(m, M)$ \demph{bounded} if its local weights are uniformly bounded on compact subsets of $X$.
\end{definition}

The following result follows directly from the definition and will be used frequently throughout the paper without further reference.

\begin{lemma}
Let \(\mathcal B\subset\overline{\mathcal H}^{\mathbb T}_2\) be a bounded family.
\begin{enumerate}[label=\textnormal{(\roman*)},leftmargin=*]
\item \(\overline{\mathcal H}^{\mathbb T}_2\) is convex and closed under \(\max\) and under
adding constants.
\item Let \(\mathcal V\) be a set of \(\mathbb T\)-invariant psh metrics
such that each \(v\in\mathcal V\) satisfies \(v^- \leq v \leq v^+\) for some
\(v^-,v^+\in\mathcal B\).  Then \(\mathcal B\cup\mathcal V\) is a bounded
family whose constants depend only on \(\mathcal B\).
\item For every \(a>0\), the set
\(\{\gamma_{-a}^*\phi:\ \phi\in\mathcal B\}\cup\mathcal B\) is a bounded
family whose constants depend only on \(a\) and \(\mathcal B\).
\end{enumerate}
\end{lemma}

The ambient metric belongs to this class, as do the canonical metrics of any shrinkers which may exist on the fibration.

\begin{lemma}\label{lem:ambient-metric-quadratic-growth}
    The metric $\psi_{\mathrm{amb}}$ defined in \cref{eq:defpsiamb} belongs to $\cH_2^\bbT$.
\end{lemma}

\begin{proof}
    Since $h_{\mathrm{AC}} = r^2/2$ outside a compact set, \cref{eq:cone-potential-flow} and compactness of $\Sigma_0$ give
    \[
        (\gamma_s^*\psi_{\mathrm{amb}} - \psi_{\mathrm{amb}})|_{\Sigma_0}
        = \frac12 e^s + O(s)
        \p
    \]
    Here the $O(s)$ term comes from the Fubini--Study factor and is uniform on $\Sigma_0$. After adjusting the constants, this gives the bounds in \cref{eq:H2-bounds}; invariance under $\bbT$ is part of the construction.
\end{proof}

\begin{lemma}
    \label{lem:psh-weight-growth}
    Let $(\omega, f)$ be a shrinker with Reeb vector field $\xi$. Then the associated Hermitian metric $\phi$, defined by $e^{-\phi}=e^{-f}\omega^n$, belongs to $\cH_2^\bbT$.
    More precisely, there are
    $0<m \leq M<\infty$ such that
    \begin{equation}\label{eq:psh-weight-growth}
        m e^s - m^{-1}
        \leq (\gamma_s^*\phi - \phi)|_{\Sigma_0}
        \leq M e^s,
        \qquad s \geq 1
        \p
    \end{equation}
\end{lemma}

\begin{proof}
    Set $\rho^2 = f + n$. Since $r^2(\gamma_s y) = e^s$ for $y \in \Sigma_0$, \cref{lem:affine-return-time} and compactness of the two transversals show that
    \[
        c e^s \leq \rho^2(\gamma_s y) \leq C e^s,
        \qquad s \geq 0, \quad y \in \Sigma_0
        \p
    \]
    Thus the shrinker radius and the cone radius are asymptotic in the sense that
    \[
        C^{-1} r^2 \leq \rho^2 \leq C r^2
    \]
    outside a compact set. The trace of the soliton equation gives
    \begin{equation}\label{eq:volume-flow}
        \cL_{X_0}\omega^n = (n - R_\omega)\omega^n
        \p
    \end{equation}
    Together with $X_0 f = \rho^2 - R_\omega$, this yields
    \begin{equation}\label{eq:density-evolution}
        \cL_{X_0}(e^{-f}\omega^n)
        = -f e^{-f}\omega^n
        \p
    \end{equation}
    Hence
    \begin{equation}\label{eq:weight-evolution}
        \frac{d}{d s}\gamma_s^*\phi = \gamma_s^*f
        \p
    \end{equation}
    In particular,
    \begin{equation}\label{eq:weight-integral}
        (\gamma_s^*\phi - \phi)(y)
        = \int_0^s f(\gamma_q y)\,dq
    \end{equation}
    is an identity of globally defined functions. Since $f = \rho^2 - n$, integration in \cref{eq:weight-integral} proves \cref{eq:psh-weight-growth}, after adjusting the constants.
\end{proof}

\begin{lemma}[Uniform asymptotics along the flow]
    \label{lem:flow-quadratic-control}
    Let $\cK$ be a bounded family in $\overline\cH_2^{\bbT}(m,M)$.
    There are $a, b > 0$, uniform for every ordered pair $\psi_-, \psi_+ \in \cK$, such that
    \begin{equation}\label{eq:flow-quadratic-endpoint-shifts}
        \gamma_{-a}^*\psi_- - b \leq \psi_+ \leq 
         \gamma_{a}^* \psi_-+b
        \p
    \end{equation}
   Moreover there  exists $C, c > 0$, uniform for $\psi \in \cK$, such that for $s\geq 0$
    \begin{equation}\label{eq:flow-quadratic-control}
        c (e^s-1) r^2-Cs \leq \gamma_s^*\psi - \psi \leq C (e^s-1)r^2+Cs
        \p
    \end{equation}
\end{lemma}

\begin{proof}
    We first prove \cref{eq:flow-quadratic-endpoint-shifts}. Choose $a \geq 1$ so large that $e^{-a}M < m/2$. For $s \geq a + 1$ and $y \in \Sigma_0$, \cref{eq:H2-bounds} gives
    \[
        \begin{aligned}
            (\gamma_{-a}^*\psi_- - \psi_+)(\gamma_s y)
                &= (\psi_- - \psi_+)(y) + (\gamma_{s - a}^*\psi_- - \psi_-)(y) - (\gamma_s^*\psi_+ - \psi_+)(y)\\
                & \leq (\psi_- - \psi_+)(y) + m^{-1} - \frac12 m e^s.
        \end{aligned}
    \]
    The first term is uniformly bounded on $\Sigma_0$, and the ordered
    difference is uniformly bounded above on the compact complement of the
    end because $\cK$ is bounded. Thus
    $\gamma_{-a}^*\psi_- - b \leq \psi_+$, with $b$ independent of the ordered pair. Interchanging $\psi_-$ and $\psi_+$ proves \cref{eq:flow-quadratic-endpoint-shifts}.

  We prove \cref{eq:flow-quadratic-control} by bounding the rate of change of the weight and then integrating along the flow. Fix $\psi \in \cK$ and $x \in X$. We know that $t \mapsto (\gamma_t^*\psi)(x)$ is finite and convex by plurisubharmonicity and $\bbT$-invariance and therefore we have for any $a>0$
  \begin{equation}\label{eq-convexity conttol}
\frac{(\gamma_t^*\psi - \gamma_{t - a}^*\psi)(x)}{a}
        \leq (\gamma_t^*\psi)'_-(x) \leq (\gamma_t^*\psi)'_+(x)
        \leq (\gamma_{t + 1}^*\psi - \gamma_t^*\psi)(x).   
  \end{equation}
Using the group law, we have 
\begin{equation}\label{eq-estimate without t}
    \frac{(\gamma_t^*\psi - \gamma_{t - a}^*\psi)(x)}{a}=\gamma_t^*\left(\frac{\psi-\gamma_{-a}^*\psi}{a}\right)(x),\quad (\gamma_{t + 1}^*\psi - \gamma_t^*\psi)(x)=\gamma_t^*(\gamma_1^*\psi-\psi)(x)
    \p
\end{equation}
We then show that there are constants $c, C > 0$, uniform over $\cK$, such that
\begin{equation}\label{eq-control without t}
    c r^2 - C \leq \frac{\psi - \gamma_{-a}^*\psi}{a},
        \qquad
        \gamma_1^*\psi - \psi \leq C(1 + r^2)
        \qquad\text{on }X
    \p
\end{equation}
   We first establish them on the end, write $x = \gamma_t y$ with $y \in \Sigma_0$ and $t \geq a + 1$, where $a \geq 1$ is fixed as before satisfying $e^{-a}M < m/2$. The group law and \cref{eq:H2-bounds} give
    \begin{equation}\label{eq--effective control}
          \begin{aligned}
            (\psi - \gamma_{-a}^*\psi)(x)
                &= (\gamma_t^*\psi - \gamma_{t - a}^*\psi)(y)
                \geq (m - M e^{-a})r^2(x) - m^{-1},\\
            (\gamma_1^*\psi - \psi)(x)
                &= (\gamma_{t + 1}^*\psi - \gamma_t^*\psi)(y)
                \leq e M r^2(x) + m^{-1}.
        \end{aligned} 
    \end{equation}
    The first coefficient is positive because $M e^{-a} < m/2$. In the compact region $\{r^2 \leq e^{a + 1}\}$, both differences are uniformly bounded: the boundedness of $\cK$ applies to a fixed compact set containing $x$, $\gamma_{-a}x$, and $\gamma_1 x$. Enlarging $C$ therefore proves the two bounds on all of $X$.

   Combining \cref{eq-convexity conttol,eq-control without t,eq--effective control} and using $r^2(\gamma_t x) = e^t r^2(x)$, we obtain
    \[
        c e^t r^2(x) - C
        \leq (\gamma_t^*\psi)'(x)
        \leq C e^t r^2(x) + C
        \quad\text{for almost every }t \geq 0
        \p
    \]
    Integrating over $[0, s]$ for any $s \geq 0$ then gives \cref{eq:flow-quadratic-control}.
\end{proof}

\begin{lemma} \label{thm:H2moments}
    For every $N$, the polynomial moment
    \[
        \int_X (1 + r^2)^N e^{-\psi}
    \]
    is uniformly bounded for $\psi$ in a bounded family in $\overline{\cH}_2^{\bbT}(m, M)$.
\end{lemma}

\begin{proof}
    Let $\Gamma(s, y) = \gamma_s y$ and fix a smooth positive density $d\sigma_0$ on $\Sigma_0$. Since the family is bounded on the compact set $\Sigma_0$, the densities $(\iota_{X_0}e^{-\psi})|_{\Sigma_0}$ are uniformly bounded above by $C d\sigma_0$. Moreover,
    \[
        \Gamma^*(e^{-\psi}) = e^{- (\gamma_s^*\psi - \psi)(y)} \d s \wedge (\iota_{X_0}e^{-\psi})|_{\Sigma_0}(y)
        \p
    \]
    The lower bound in \cref{eq:H2-bounds} therefore gives the uniform estimate
    \begin{equation}\label{eq:common-gaussian-density}
        \Gamma^*(e^{-\psi}) \leq C\exp(-m e^s)\,ds\,d\sigma_0,
        \qquad s \geq 1
        \p
    \end{equation}
    Since $r^2(\gamma_s y) = e^s$, it follows that
    \[
        \int_{\{r^2 \geq e\}}(1 + r^2)^N e^{-\psi}
        \leq C\int_1^\infty (1 + e^s)^N\exp(-m e^s)\,ds < \infty
        \p
    \]
    Boundedness of the family on the compact complement gives a uniform bound for the remaining integral.
\end{proof}

\subsection{Construction of the geodesic}

Write $\tau = t + \sqrt{-1}\sigma$ and let
$\cS=\{0<t<1\}\subset\bbC$.

\begin{lemma}[subgeodesic segments]
    \label{lem:psh-subgeodesic}
    Let $\psi_-,\psi_+$ be continuous $\bbT$-invariant locally bounded psh
    metrics in one quadratic-growth class, and choose $a,B$ from
    \cref{eq:flow-quadratic-endpoint-shifts}.  On a neighborhood of
    $\overline\cS\times X$, set
    \begin{equation}\label{eq:psh-subgeodesic}
        \underline\Psi_\tau
        =\max\left\{
            \gamma_{-a\tau}^*\psi_- -Bt,
            \gamma_{-a(1-\tau)}^*\psi_+ -B(1-t)
        \right\}
        \p
    \end{equation}
    Then $\underline\Psi$ is a continuous, locally bounded,
    $\bbT$-invariant psh metric on $\operatorname{pr}_X^*K_X^{-1}$.  It is
    independent of $\sigma = \Im \tau$ and has boundary values
    $\underline\Psi_0=\psi_-$ and $\underline\Psi_1=\psi_+$.
\end{lemma}

\begin{proof}
    The maps $(\tau,x)\mapsto\gamma_{-a\tau}x$ and $(\tau,x)\mapsto\gamma_{-a(1-\tau)}x$ are holomorphic and $\bbT$-equivariant. Their canonical lifts to $K_X^{-1}$ identify the pulled-back anticanonical bundles with $\operatorname{pr}_X^*K_X^{-1}$, so both branches in \cref{eq:psh-subgeodesic} are continuous locally bounded psh metrics on the same line bundle. Hence their maximum is also a psh metric on that line bundle. Furthermore, since $\psi_-, \psi_+$ are $\bbT$-invariant, both branches are independent of $\Im \tau$ and $\bbT$-invariant. Thus the same is true of $\underline\Psi$.

    The inequalities \cref{eq:flow-quadratic-endpoint-shifts} precisely say that $t = 0$ (resp $t = 1$) the first (resp. second) argument of $\max$ achieves the maximum, which shows that $\underline\Psi$ has the claimed boundary values.
\end{proof}

\begin{defn}
    A \demph{geodesic} connecting metrics $\psi_0, \psi_1 \in \overline{\cH}_2^{\bbT}$ is a locally bounded psh metric $\Phi \in \PSH(\cS \times X; \pr_2^*K_X^{-1})^{\bbT}$, constant in the $\sqrt{-1}\bbR$-direction, such that there exist $m, M$ for which $\Phi_t \in \overline{\cH}_2(m, M)^{\bbT}$ for all $t$, and which solves the boundary value problem
    \[
        \begin{cases}
            \left(\sqrt{-1}\partial_{\tau, X}\bar\partial_{\tau, X}\Phi\right)^{n + 1} = 0, \\
            \lim_{t \to i}\Phi_t = \psi_i, \quad i = 0, 1.
        \end{cases}
    \]
    The limit is in the pointwise sense. A \demph{subgeodesic} has the same definition, except we only require $(\sqrt{-1}\partial_{\tau, X}\bar\partial_{\tau, X}\Phi)^{n + 1} \geq 0$.
\end{defn}

\begin{proposition}
    \label{prop:general-geodesic}
    Let $\psi_0,\psi_1$ be continuous, $\bbT$-invariant psh metrics in one quadratic-growth class. Then there is a $\bbT$-invariant locally bounded geodesic $\Phi_t$, $0 \leq t \leq 1$ whose slices $\{\Phi_t\} \subseteq \overline{\cH}_2$ are a bounded family, and converge locally uniformly to $\psi_i$ as $t \to i = 0, 1$.
    Furthermore $\Phi$ is Lipschitz in $t$ with quadratic growth, i.e.\ $\exists L > 0$ such that
    \begin{equation}\label{eq:general-geodesic-bound}
        |\Phi_t - \Phi_u| \leq L(1 + r^2)|t - u|
        \p
    \end{equation}
\end{proposition}

\begin{proof}
    Choose $m,M$ such that $\psi_0, \psi_1 \in \overline{\cH}_2(m, M)$ and apply \cref{lem:flow-quadratic-control} to the bounded two-element
    family $\{\psi_0,\psi_1\}$.  Let $a,B$ be its endpoint-shift constants and
    let $\underline\Psi$ be the subgeodesic from \cref{lem:psh-subgeodesic}.

    Let $\mathcal C$ be the class of locally bounded psh metrics $\Psi$ on
    $\operatorname{pr}_X^*K_X^{-1}\to\cS\times X$ which are independent of
    $\sigma = \Im \tau$ and satisfy
    \[
        \begin{aligned}
            \limsup_{(\tau,x')\to(0,x)}\Psi(\tau,x')& \leq \psi_0(x),\\
            \limsup_{(\tau,x')\to(1,x)}\Psi(\tau,x')& \leq \psi_1(x).
        \end{aligned}
    \]
    We now apply Perron's method with $\cC$ as our set of competitors, i.e.\ the geodesic is the upper semicontinuous regularization
\begin{equation}\label{eq:perron-envelope}
        \Phi = \left(\sup_{\Psi\in\mathcal C}\Psi\right)^*
        \p
    \end{equation}
    In order for $\Phi$ to be well-defined we need that $\cC$ is nonempty, which we know since $\underline{\Psi} \in \cC$ and that $\cC$ is locally uniformly bounded above, which we show now:
    for fixed $x \in X$, picking a trivialization of the fiber $K_X^{-1}|_x$, the restriction of a competitor $\Psi \in \cC$ to $\cS \times \{x\}$ is a convex function of $t$.  Hence we have the "chord bound"
    \begin{equation}\label{eq:chord-bound}
        \Psi_t(x) \leq (1-t)\psi_0(x)+t\psi_1(x)
        \p
    \end{equation}
    Thus the competitors are locally uniformly bounded above, and the upper-semicontinuous regularization is a locally bounded psh metric by the Choquet lemma, bounded below by $\underline\Psi$.
    Furthermore, since $\psi_0$ and $\psi_1$ are $\bbT$-invariant, the set $\cC$ is preserved by $\bbT$, and hence $\Phi$ is $\bbT$-invariant.

    The usual local maximality argument shows that $\Phi$ satisfies the homogeneous Monge--Ampère equation: On any coordinate ball $B$, $\Phi$ is the maximal locally bounded PSh function with boundary values bounded above by $\Phi|_{\partial B}$. Indeed, if some other PSh function $U$ was larger than $\Phi$, we can glue $\Phi$ and $\max \{\Phi, U - \epsilon\}$ to create a locally bounded PSh metric on $\operatorname{pr}_X^* K_X^{-1}$ which is not bounded above by $\Phi$. Then taking the upper-semicontinuous regularization of all translates in $\Im \tau$-direction of that metric produces a competitor in $\cC$ which exceeds $\Phi$, contradicting the maximality. Finally, by the Bedford--Taylor characterization of solutions to the MA-Dirichlet problem \cite[Cor. 5.15]{GuedjZeriahi}, local maximality implies that
    \begin{equation}\label{eq:flow-quadratic-hcma}
        \left(\sqrt{-1}\partial_{\tau,X}
        \bar\partial_{\tau,X}\Phi\right)^{n+1}=0
        \p
    \end{equation}
    See \cite[Thm.~2.9]{EAC} for details.

    We next verify that $\{\Phi_t\}$ is a bounded family.
    First note that the chord bound \cref{eq:chord-bound} and $\Phi \geq \underline{\Psi}$ show that $\Phi_t$ is uniformly locally bounded. Thus we only need to prove the uniform asymptotic bound \cref{eq:psh-weight-growth} for large $s$.

    For $s\geq a+1$ and $y\in\Sigma_0$, use the first subgeodesic branch at $\gamma_s y$ and the chord bound at $y$ to obtain
    \[
        (\gamma_s^*\Phi_t-\Phi_t)(y)
            \geq \gamma_s^* \underline{\Psi}_t - C
            \geq \gamma_s^* \gamma_{-at}^* \psi_0 - C
            \geq m e^{s-at}-C\geq m e^{-a}e^s-C
        \p
    \]
    Similarly, using the chord bound at $\gamma_sy$ and the bound $(\gamma_s^* \psi_{0, 1} - \psi_{0, 1})|_{\Sigma_0} \leq M e^s$ we obtain
    \[
        (\gamma_s^*\Phi_t-\Phi_t)(y) \leq M e^s+C
        \p
    \]

    It remains to prove the quadratic Lipschitz estimate \cref{eq:general-geodesic-bound}.
    Since $\Phi_t$ is convex in $t$-direction, this amounts to proving the quadratic bound
    \[
        \left|\frac{d^\pm}{d t}\Phi_t\right| \leq L(1 + r^2)
    \]
    for the left and right derivative. Note that
    \begin{align}
        \frac{d^\pm}{d t} \Phi_t
            & \leq \frac{\psi_1 - \Phi_t}{1 - t} \leq \frac{\psi_1 - \underline{\Psi}_t}{1 - t}
                \tag{by $t$-convexity and $\underline{\Psi} \leq \Phi$}\\
            & \leq \frac{\psi_1 - \gamma_{-a(1-t)}^* \psi_1 + B(1-t)}{1-t}
                \tag{right branch for $\underline{\Psi}$} \\
            & \leq B + \frac{\gamma_{a h}^* \psi_1 - \psi_1}{h} \quad \text{for any $h > 0$,}
                \tag{since $s \mapsto \gamma_{as}^* \psi_1$ is convex}\\
            & \leq B + C e^a(1 + r^2)
                \tag{pick $h = 1$}
    \end{align}
    which proves the claimed upper bound on $d^+ \Phi_t / dt$. Swapping $\psi_1$ and $\psi_0$ then proves the analogous lower bound for $d^- \Phi_t / dt$

    Finally, squeezing $\Phi_t$ above by \cref{eq:chord-bound} and below by $\underline{\Psi}_t$ we see that $\Phi_t \to \psi_i$ locally uniformly as $t \to i = 0, 1$ because both the upper and lower bound converge locally uniformly to $\psi_i$.
\end{proof}

\subsection{Weighted Monge--Amp\`ere measure and energy}

All metrics considered in this section lie
in a bounded family of some \(\overline{\cH}_2(m,M)^{\bbT}\).
\Cref{lem:global-regularization} and the uniform tail estimates allow the
standard partial-Legendre and weighted-energy machinery of
\cite[Section~3]{EAC} to be used in this class.

Let \(Z_\alpha\) be the connected components of \(\{X_0=0\}\).
We say that the \emph{critical value} of $Z_\alpha$ is
\[
    \lambda_{Z_\alpha} = h^\psi|_{Z_\alpha},
\]
which only depends on $\xi$ but not on the choice of metric $\psi$.
\[
    \lambda_0 = \min_\alpha\lambda_{Z_\alpha},
        \qquad
    Z_\alpha^- = \{x:\lim_{t\to\infty}\gamma_{-t}x\in Z_\alpha\},
        \qquad
    A_\lambda = \bigcup_{\lambda_{Z_\alpha}\geq\lambda}Z_\alpha^- \p
\]

Just as in \cite[Prop.~3.11]{EAC} the convexity of a smooth $\psi \in \cH_2^\bbT$ along the flow of $X_0$ shows
\begin{lemma} \label{thm:HamiltonianBounds}
    There exists $C = C(m, M)$ such that on $r^2 \geq 1$
    \[
        C^{-1} r^2 - C \leq h^\psi \leq C r^2
    \]
    for every $\psi \in \cH_2(m, M)$
\end{lemma}

The arguments in \cite[Section 3.1]{EAC}, which generalize \cite{BermanWittNystrom} to the non-compact setting, apply verbatim in our general setting.
We recall the relevant definitions and their properties here.
For $\lambda > \lambda_0$, define the \emph{partial Legendre transform} of a metric $\psi \in \overline{\cH}_2^\bbT$
\[
    P_\lambda \psi = \inf_{t\geq0}\{\gamma_{-t}^*\psi+\lambda t\}
    \p
\]

With the same proof as \cite[Lem.~3.14]{EAC} one shows
\begin{lemma}
    $P_\lambda \psi$ is psh and is locally bounded on $X\setminus A_\lambda$.
\end{lemma}

Thus on $X \setminus A_\lambda$ we can define the Monge--Ampère measure $\MA(P_\lambda \psi)$ using the Bedford--Taylor definition. We extend this to a measure on $X$ by pushforward along the inclusion $X \setminus A_\lambda \subseteq X$. From now on $\MA(P_\lambda \psi)$ will refer to the measure on $X$.
The next lemma shows that $\MA(P_\lambda \psi)$ has locally finite (in fact, finite) mass so our definition of $\MA(P_\lambda \psi)$ agrees with the non-pluripolar Monge--Ampère measure of \cite{BEGZ}.
\begin{lemma}
    For $\psi \in \overline{\cH}_2^\bbT$ and $\lambda > \lambda_0$ the following hold
    \begin{enumerate}
        \item $\MA(P_\lambda \psi)$ has finite total mass, independent of $\psi$.
        \item $\psi \mapsto \MA(P_\lambda \psi)$ is weakly continuous along decreasing $\psi$.
        \item On $\{r^2 > C(1 + \max\{\lambda, 0\})\}$ we have
            \[
                P_\lambda\psi(\gamma_s x) = P_\lambda\psi(x) + \lambda s \quad \forall s \geq 0
            \]
            Consequently,
            \begin{equation}\label{eq:legendre-ma-support}
                \operatorname{supp}\operatorname{MA}(P_\lambda\psi) \subset \{r^2 \leq C(1 + \max\{\lambda, 0\})\};
            \end{equation}
        \item $\nu \geq \lambda$ implies
            $\operatorname{MA}(P_\nu\psi)\geq\operatorname{MA}(P_\lambda\psi)$.
        \item For smooth $\psi\in \mathcal{H}_2^{\mathbb T}$,
            \[
                \{P_\lambda \psi = \psi\}=\{h^\psi \leq \lambda\},
                    \qquad
                \operatorname{MA}(P_\lambda\psi) = \mathbf1_{\{h^\psi<\lambda\}} (\sqrt{-1}\partial\bar\partial\psi)^n
                \p
            \]
    \end{enumerate}
\end{lemma}

The definition of the $\xi$-weighted Monge--Ampère measure in \cite[\S 4.2]{EAC} only requires one small change: Since we are not assuming that the polarized Fano fibration $X$ admits a shrinker, we cannot use a shrinker as a background metric as was done there. Instead we will use the ambient metric $\psi_\mathrm{amb}$.
The necessary estimate is
\begin{lemma}[Tail estimate for ambient weighted volume] \label{thm:GaussianMoments}
  For any $\tau>0$,  there exists an $\epsilon > 0$ such that
    \[
        \int_{X \setminus \{r \leq R\}} e^{-\tau h^{\psi_\mathrm{amb}}} \omega_\mathrm{amb}^n = O\left(e^{-\epsilon R^2}\right)
    \]
    as $R \to \infty$.
    In particular, all polynomial moments of the measure $e^{-h^{\psi_\mathrm{amb}}} \omega_\mathrm{amb}^n$ are finite.
    \begin{equation}\label{eq:gaussian-integrability}
        \int_X (1 + r^2)^N e^{-h^{\psi_{\mathrm{amb}}}}\omega_{\mathrm{amb}}^n < \infty
        \p
    \end{equation}
\end{lemma}

\begin{proof}
    As in \cref{thm:H2moments} we can embed $X \setminus \{r \leq R\}$ into a product $\Sigma_0 \times [2\log R-C, \infty)$ via $\Gamma(s, y) = \gamma_s y$.
    Equation \cref{eq:section-four-ambient-flow-bound} gives $\gamma_s^*\omega_{\mathrm{amb}}^n \leq e^{C s}\omega_{\mathrm{amb}}^n$ and by \cref{thm:HamiltonianBounds} we have $h^\psi \geq -C + C^{-1} e^s$. Thus the integrand in \cref{eq:gaussian-integrability}, expressed on $\Sigma_0 \times [\log R, \infty)$, is bounded by
    \[
        C\exp\bigl((N + C)s - C^{-1} e^s\bigr) \sigma \wedge \d s
    \]
    where $\sigma$ is some measure on $\Sigma_0$. We obtain
    \begin{align}
        \int_{\{r \geq R\}} e^{-h^{\psi_\mathrm{amb}}} \omega_\mathrm{amb}^n
            & \leq  C \int_{\Sigma_0} \sigma \int_{2\log R-C}^\infty \exp( C s - C^{-1} e^s) \d s \\
            &= O\left(e^{-C^{-1} R^2}\right)
        \p
    \end{align}
    The bound on the moments can be obtained through
    \[
        \begin{aligned}
            \int_X (1 + r^2)^N e^{-h^{\psi_{\mathrm{amb}}}}\omega_{\mathrm{amb}}^n
                & = \int_X e^{-h^{\psi_{\mathrm{amb}}}}\omega_{\mathrm{amb}}^n \\
                &\quad + 2N\int_0^\infty R(1 + R^2)^{N - 1}\left(\int_{\{r \geq R\}} e^{-h^{\psi_{\mathrm{amb}}}}\omega_{\mathrm{amb}}^n\right)\,\d R.
        \end{aligned}
    \]
\end{proof}

Define $\operatorname{MA}_{\mathbf1_{[\lambda_0,\lambda]}}(\psi) := \operatorname{MA}(P_\lambda\psi)$, and extend this to $\operatorname{MA}_g(\psi)$ for nonnegative step functions $g$ by linearity in $g$.
Then by (1) the mass of $\MA_g(\psi)$ is independent of $\psi$, so we can use the smooth metric $\psi_{\mathrm{amb}}$ to compute it using (5):
\[
    \int_X \operatorname{MA}_g(\psi) = \int_X g(h^{\psi_{\mathrm{amb}}})\omega_{\mathrm{amb}}^n,
    \p
\]

\begin{lem}\label{lem:total mass}
    For every sequence of step functions $g_k: [\lambda_0, \infty) \to \bbR_{>0}$ increasing to $e^{-x}$ the limit
    \begin{equation}
        \MA_\xi \coloneq \lim_{k \to \infty} \MA_{g_k}(\psi)
    \end{equation}
    exists and is independent of the sequence $g_k$. Furthermore $\MA_\xi(\psi) = e^{-h^\psi} (i \partial\bar{\partial} \psi)^n$ for smooth $\psi \in \cH_2^\bbT$ and
    \[
     \int_X \MA_\xi (\psi) < \infty
    \]
    is independent of $\psi \in \overline{\cH}_2^\bbT$.
\end{lem}
\begin{proof}
    This is proven just like \cite[Prop.~3.12]{EAC}, using $\psi_\mathrm{amb}$ as the smooth reference metric. The relevant estimate for the smooth reference metric is $\int_X \exp(-h^{\psi_\mathrm{amb}}) \MA(\psi_\mathrm{amb}) < \infty$, which was established in \cref{thm:GaussianMoments}.
\end{proof}

By \cite[Proposition 5.9]{SunZhangFano}, we know that the total mass coincides with the \textit{weighted volume} defined in \cite{SunZhangFano}:
\begin{equation}
    \mathbb W_{X}(\xi)=\frac{1}{(2\pi)^n n!}\int_X \MA_\xi (\psi)
    \p
\end{equation}
\begin{lem}[Tail estimate for weighted MA] \label{thm:MAxitail}
    For $\psi \in \overline{\cH}_2(m, M)^\bbT$ we have
    \[
        \int_{X \setminus \{r \leq R\}} (1 + r^2)^N \MA_\xi(\psi) \leq C e^{-\epsilon R^2}
    \]
    for some $C = C(m, M, N), \epsilon = \epsilon(m, M, N)$.
\end{lem}
\begin{proof}
    This is analogous to \cite[Eq.~3.12, 3.14]{EAC}, using \cref{thm:GaussianMoments} instead of the estimate for an AC metric used there.
\end{proof}

As a consequence of the tail estimate, we also obtain the continuity conclusion of \cite[Thm.~3.22]{EAC}

\begin{thm}
    If $\psi_j, \phi_j, \theta_j \in \overline{\cH}_2^\bbT$ are monotone sequences converging pointwise to limits $\psi_\infty, \phi_\infty, \theta_\infty \in \overline{\cH}_2^\bbT$ then
    \[
        (\psi_j - \phi_j) \MA_\xi(\theta_j) \weakto (\psi_\infty - \phi_\infty) \MA_\xi(\theta_\infty)
    \]
    and their integrals also converge.
\end{thm}

\begin{definition}
    For $\psi_0, \psi_1 \in \overline{\mathcal H}_2^{\mathbb T}$, define
    \begin{equation}\label{eq:weighted-energy}
        E_\xi(\psi_1,\psi_0) =\int_0^1 \int_X (\psi_1-\psi_0) \operatorname{MA}_\xi((1-t)\psi_0+t\psi_1) \d t
        \p
    \end{equation}
    This is finite by \cref{thm:MAxitail}.  Write
    \(
        E_{\xi}(\psi) = E_{\xi}(\psi, \psi_{\mathrm{amb}}),
    \)
    and define
    \begin{equation}\label{eq:modified-ding}
        \cD_{\xi}(\psi) = \mathcal{L}(\psi) - \frac{1}{(2\pi)^n n! \mathbb W_{X}(\xi)}E_{\xi}(\psi),\qquad \mathcal{L}(\psi) = -\log\int_X e^{-\psi}
        \p
    \end{equation}
\end{definition}

The proof of \cite[Proposition~3.27(ii)]{EAC} applies verbatim with the cone potential \(r^2\) introduced above.
Then using the above estimate, we can argue as in \cite[Section 3.3]{EAC} to obtain the following properties for weighted energy.
\begin{proposition}\label{prop:weighted-energy}
    Let $E_\xi$ be the weighted energy defined in \cref{eq:weighted-energy}. It satisfies the following.
    \begin{enumerate}
        \item For every smooth path $\psi_t$, $0 \leq t \leq 1$, in a bounded family of some $\cH_2^{\bbT}(m, M)$, with $|\dot \psi_t| \leq C(1+r^2)$,
            \begin{equation}\label{eq:weighted-energy-path}
                E_\xi(\psi_1,\psi_0) = \int_0^1\!\int_X\dot\psi_t\operatorname{MA}_\xi(\psi_t)\,dt
                \p
            \end{equation}
        \item The energy is continuous in both arguments along monotone sequences.
        \item $t \mapsto E_\xi(\psi, t \phi + (1-t)\psi)$ is concave, for $\phi, \psi\in \overline{\cH}_2^{\bbT}$.
        \item $E_\xi(\Phi_t,\Phi_0)$ is convex along every subgeodesic $\Phi_t\in \overline{\cH}_2^{\bbT}(m,M)$ satisfying $|\partial^\pm_t\Phi_t| \leq C(1 + r^2)$.
        \item For $\psi, \phi, \theta$ in $\overline{\cH}_2^{\bbT}$,
            \begin{equation}\label{eq:weighted-energy-cocycle}
                E_\xi(\psi, \theta) = E_\xi(\psi, \phi) + E_\xi(\phi, \theta)
                \p
            \end{equation}
    \end{enumerate}
\end{proposition}

\begin{proof}
    Parts (1)--(4) were proven in \cite{EAC} and the proofs there also work in our setting.
    For the reader's convenience we point to the specific proofs:
    \begin{enumerate}
        \item The proof of \cite[Prop.~3.26]{EAC} applies.

        \item This follows from \zcref{lem:global-regularization}, the tail estimate \zcref{thm:MAxitail} and the capacity convergence theorem as in \cite[Thm.~3.22]{EAC}.
        \item For smooth metrics which agree outside of a compact set, this follows from (1). In general, we use \cref{lem:global-regularization} and (2).
        \item Just as in \cite[Prop.~3.27(i)]{EAC} this follows by taking smooth approximations to $\Phi_t^\nu \searrow \Phi_t$ using \cref{lem:global-regularization} and differentiating \cref{eq:weighted-energy-path} twice. Differentiation under the integral sign is easy to justify because for every $\nu$, $\Phi^\nu_t$ is constant in $t$ outside a compact set.
        \item
            This follows essentially from (1).  Using the regularization with upper barrier \cref{lem:global-regularization} and the continuity of $E_\xi$ along monotone sequences, it is enough to prove the case, where $\psi,\phi,\theta\in\cH_2^\bbT$ and agree outside a compact set.
            Observe that by (1) we have $\partial_u E_\xi(\eta_u, \phi) = \partial_u E_\xi(\eta_u, \theta)$ for any smooth family $\eta_u \in \cH_2^\bbT$.
            Now put
            \[
                \eta_{s,t} = \theta + s(\phi - \theta) + t (\psi - \theta)
                \quad (s,t)\in D:=\{s,t\geq0:\ s+t \leq 1\}
                \p
            \]
            Then
            \[
                E_\xi(\phi, \theta) - E_\xi(\psi, \theta) = \int_0^1 \!\d u \, \partial_u E_\xi(\eta_{u, (1-u)}, \theta)
                                      = \int_0^1 \!\d u \, \partial_u E_\xi(\eta_{u, (1-u)}, \psi) = E_\xi(\phi, \psi)
                  \p
            \]
    \end{enumerate}
\end{proof}

\begin{remark}
      There is another way to prove \cref{eq:weighted-energy-path} for $C^1$ geodesics $\Phi_t$.
            First consider any two $\phi_0, \phi_1 \in \overline\cH_2^\bbT$ and put $\phi_t = t \phi_1 + (1-t) \phi_0$.
            By concavity in (3) we have
            \[
                E_\xi(\phi_1, \phi_0) \leq \left. \frac{d}{d t} \right|_0 E_\xi(\phi_t, \phi_0) = \int_X (\phi_1 - \phi_0) \MA_\xi(\phi_0)
                \p
            \]
            The equality is easily seen from the definition of $E_\xi$. Thus
            \[
                E_\xi(\Phi_{t + h}, \Phi_t) \leq \int_X (\Phi_{t + h} - \Phi_t) \MA_\xi(\Phi_t)
            \]
            and taking the limit gives $\frac{d^+}{d t} E_\xi(\Phi_t) \leq \int_X \dot{\Phi}^+_t \MA_\xi(\Phi_t)$. At $t \in (0, 1)$ we also have the reverse inequality $\frac{d^-}{d t} E_\xi(\Phi_t) \geq \int_X \dot{\Phi}^-_t \MA_\xi(\Phi_t)$, so $C^1$-regularity of $\Phi_t$ and the fundamental theorem of calculus prove the claim.
\end{remark}

\begin{proposition}
    \label{prop:energy-endpoint-variation}
    Let $\psi_0, \psi_1 \in \cH_2^\bbT$, and let $\Phi_t$, $0 \leq t \leq 1$, be the geodesic joining them supplied by \cref{prop:general-geodesic}. Then
    \[
        E(t) := E_\xi(\Phi_t, \psi_0)
    \]
    is affine on \([0, 1]\), and its constant slope satisfies
    \begin{equation} \label{eq:energy-endpoint-inequalities}
        \int_X \partial_t^+ \Phi_0 \operatorname{MA}_\xi(\psi_0)
            \geq \partial^+_t E(0) = \partial^-_t E(1)
            \geq \int_X \partial_t^- \Phi_1 \operatorname{MA}_\xi(\psi_1)
        \p
    \end{equation}
\end{proposition}
\begin{proof}
    The same proof as \cite[Prop.~3.27(ii)]{EAC} shows that $E(t)$ is affine. The auxiliary metric $\psi \in (\cH_2^*)^\bbT$ in \textit{loc.cit.} needs to be replaced with $\psi_\mathrm{amb}$ and we use the tail estimate \Cref{thm:MAxitail} in place of \cite[(3.14)]{EAC}.
    Note that $\dot{\Phi}^\nu \equiv 0$ outside of a compact set does not avoid the somewhat complicated tail estimates, since the proof requires picking $\nu \geq \nu(R)$. 

    The inequalities for the slope at the endpoints follow as in \cite[Lem.~4.7]{EAC} from the established concavity/convexity of $E_\xi$.
\end{proof}

\subsection{Berndtsson convexity}
We recall Berndtsson's convexity result \cite{Berndtsson2015} in the form stated in \cite[Section~4.2]{EAC}. Let $u$ be the canonically normalized tautological nowhere-vanishing section of $\Omega^n(X, K_X^{-1})$.

For a smooth metric $\chi$ whose curvature $\omega_\chi$ is complete, norms on forms are taken with respect to $\omega_\chi$, while subscripts indicate the weight. We set $\partial^\phi = e^\phi\partial e^{-\phi}$ and, as usual, write $\mathcal A^{p,q}(X;K_X^{-1})$ for the smooth $K_X^{-1}$-valued $(p,q)$-forms.

\begin{lemma}[{\cite[Proposition~4.3]{EAC}}]
    \label{lem:complete-primitive-solution}
    Let $\chi \in \cH_2^{\bbT}$ be smooth with $\omega_\chi$ complete. Suppose that $\phi \in \cH_2^{\bbT}(m,M)$ satisfies, for some $\kappa > 0$,
    \[
        \phi \leq \chi,
        \qquad
        \sqrt{-1}\partial\bar\partial\phi \geq \kappa\omega_\chi
    \]
    If $f \in C^\infty(X;\bbR)^{\bbT}$ satisfies
    $|f| \leq L(1 + r^2)$, then
    \[
        \partial^\phi v=\pi_\perp^\phi(fu)\coloneq \left(
        f-\frac{\int_X f e^{-\phi}}{\int_X e^{-\phi}}
    \right)u
    \]
    has a unique solution
    $v \in L^2_{\omega_\chi, \phi}\mathcal A^{n-1,0}(X;K_X^{-1})^{\bbT}$ such that
    $\bar\partial v$ is $\omega_\chi$-primitive. This solution satisfies
    \begin{equation}\label{eq:unifL2}
        \|v\|_{\omega_\chi,\phi}<C(m,M,L)
    \end{equation}
\end{lemma}
\begin{proof}
    This is proven in \cite[Proposition~4.3]{EAC}. The arguments used there to show that the quantities
    \[
        \left\lVert e^{\epsilon (\chi - \phi)} u\right\rVert_{(1-\epsilon) \phi + \epsilon \chi},\quad \left\lVert e^{\epsilon (\chi - \phi)} f u\right\rVert_{(1-\epsilon) \phi + \epsilon \chi},\quad \left\lvert \int_X e^{-\phi} f u \right\rvert
    \]
    are bounded for small $\epsilon$ still apply here. However in the style of exposition that we have used in this paper, we would rather argue that, using the notation from the proof of \cref{thm:H2moments}
    \begin{align}
        \Gamma^* \left( e^{-\phi} e^{\epsilon (\chi - \phi)} \right) \leq C \exp(\epsilon C e^s) \exp(-m e^s) \d s \wedge \d \sigma_0 < \infty \\
        \Gamma^* \left( e^{-\phi} e^{\epsilon (\chi - \phi)} f^2 \right) \leq C \exp(\epsilon C e^s) e^{2s} \exp(-m e^s) \d s \wedge \d \sigma_0 < \infty
    \end{align}
    for $\epsilon < \frac{m}{C}$.
\end{proof}

Note that on $\bbT$-invariant forms $\Omega^n(X, K_X^{-1})$ the operator $\pi_\perp^\phi$ is the $L^2(e^{-\phi})$-orthogonal projection onto the orthogonal complement of holomorphic forms.

\begin{proposition}
    \label{prop:complete-berndtsson}
    Let $\Phi_t$, $0 \leq t \leq 1$, be a $\bbT$-invariant geodesic whose slices form a bounded family in $\overline{\cH}_2^{\bbT}(m,M)$ and which satisfies
    $|\partial_t^\pm \Phi_t| \leq C(1 + r^2)$.
    Set
    \[
        \mathcal L(t) = -\log\int_X e^{-\Phi_t}
        \p
    \]
    Then $\mathcal L$ is convex and continuous on $[0,1]$.
    If, in addition, $\mathcal L$ is affine, then there is a $\bbT$-invariant holomorphic vector field $V$, whose flow we denote by $F_t: X \to X$, such that
    \[
        F_t^*\Phi_t = \Phi_0 + \mathcal L(t) - \mathcal L(0)
        \p
    \]
\end{proposition}

\begin{proof}
    Let $\Phi^j \searrow \Phi$ be the smooth approximation of $\Phi$ constructed in \cref{lem:global-regularization} and $\chi \geq \Phi^j$ the common upper bound also constructed there. For $t\in [j^{-1},1-j^{-1}]$, set
    \[
        \mathcal{L}_j(t) = -\log \int_X e^{-\Phi_t^j},
            \qquad
        m_j(t) = \exp({\mathcal{L}_j(t)}) \int_X\dot\Phi_t^j e^{-\Phi_t^j}
        \p
    \]
    Since $\omega_{\Phi_t^j}$ is Kähler and agrees with $\omega_\chi$ outside a compact set, it is bounded below by a positive multiple of $\omega_\chi$ for each $j$ and $t$.
    Applying \cref{lem:complete-primitive-solution} with $\phi = \Phi_t^j$ and $f = \dot\Phi_t^j$ gives $v_t^j$ satisfying
    \[
        \partial^{\Phi_t^j}v_t^j = (\dot\Phi_t^j-m_j(t))u,
            \qquad
        \bar\partial v_t^j\wedge\omega_\chi=0
        \p
    \]
    For $\tau = t + \sqrt{-1}\sigma$, put
    \[
        \widehat u_j=u-\frac12d\tau\wedge v_t^j
        \p
    \]
    Then, by the calculation in \cite{Berndtsson2015} or \cite[Thm.~4.5]{EAC}, together with the identity $\partial_\tau \Phi^j=\dot{\Phi}^j/2$, we obtain
    \begin{equation}\label{eq:berndtsson-formula}
        \exp({\mathcal{L}_j(t)})\mathcal{L}_j''(t)
            = 4\sqrt{-1}\, \iota_{\partial/\partial\tau} \iota_{\partial/\partial\bar\tau} \int_X \big(\sqrt{-1}\partial\bar\partial\big)_{\tau,X}\Phi^j \wedge i^{n^2} \widehat u_j\wedge\overline{\widehat u_j}\,e^{-\Phi^j}
              + \|\bar\partial v_t^j\|_{\omega_\chi,\Phi_t^j}^2
    \end{equation}
    In fact, the computation in \cite[Appendix C]{EAC} can be simplified with our approximation theorem because $\dot{\Phi}_t^j$ is compactly supported, and $i \partial\bar{\partial} \Phi^j_t \equiv \omega_\chi$ outside of a compact set. This makes it easy to see that all integrals appearing in the computation are finite.

    \emph{Convexity.}
    Joint plurisubharmonicity of $\Phi^j$ shows that the first term on the RHS of \cref{eq:berndtsson-formula} is positive, so $\mathcal{L}_j$ is convex on $J_j$. Since $\Phi^j$ decrease to $\Phi$, monotone convergence shows that $\cL_j \to \cL$ pointwise. Thus $\cL$ is convex on $(0,1)$. The time-derivative bound shows $\mathcal{L}$ is continuous on $[0,1]$ and hence convex on $[0,1]$.

    \emph{Geodesics with affine $\cL$ are generated by a vector field.}
    The arguments in \cite[\S 4.4]{EAC} work without modification. They require $\fh_\xi \coloneq \aut(X, \xi)$ to be finite-dimensional, which we establish in \cref{lem:analytic-algebraic-centralizer}.
    Thus we obtain $i \partial\bar{\partial} \Phi_t = (F_t)_* i \partial\bar{\partial} \Phi_0$ where $F_t$ is the flow of a \emph{time-dependent} vector field $V_t$.
    In \cite[Lem.~6.5]{CE} it was shown that this actually implies $\Phi_t - \cL(t) = (F_t)_* \Phi_0 - \cL(0)$.
    Finally, Berndtsson shows in \cite[\S 4.1]{Berndtsson2015} that $V_t$ is constant in $t$.
\end{proof}

\begin{lemma}\label{lem:analytic-algebraic-centralizer}
Let $(\pi: X\to Y,\xi)$ be a smooth polarized Fano fibration. Then
\begin{equation}\label{eq-analytic centralizer equals algebraic}
    \fh_\xi=
\{V\in H^0(X^{\an},T_{X^{\an}}):[V,\xi]=0\}
=
\{V\in H^0(X,T_X):[V,\xi]=0\}
    \p
\end{equation}
 is finite-dimensional.
\end{lemma}

\begin{proof}
Let $R:=\mathcal O(Y)$, $\mathcal F:=\pi_*T_X.$
Since $\pi$ is projective and $\mathcal F$ is coherent, relative GAGA gives
\[
    \mathcal F^{\an}\simeq \pi_*^{\an}T_{X^{\an}}
    \p
\]
As $Y$ is affine, $M:=H^0(Y,\mathcal F)$ is a finitely generated $R$-module. Choose homogeneous $T$-equivariant generators $e_1,\ldots,e_N$ of $M$, giving a surjection $\bigoplus_{j=1}^N\mathcal O_Y e_j\longrightarrow\mathcal F$.

Let $V\in H^0(X^{\an},T_{X^{\an}})$ satisfy $[V,\xi]=0$. Since the algebraic sections $e_j$ are also generators of $\cF^\an$, we have in a $\bbT$-invariant neighborhood $U$ of the vertex $o_Y \in Y$
\[
    V|_U = \sum_{j=1}^N f_j e_j|_U,
        \qquad
    f_j \in \cO(U^\an)
    \p
\]
Since $V$ is $\xi$-invariant, averaging the coefficient tuple $(f_1,\ldots,f_N)$, we may assume $f_j$ are homogeneous under the $\xi$-action. Then we can extend the $f_j$ to all of $Y^\an$ by homogeneity and by homogeneity we also see that $V = \sum_j f_j e_j$ on all of $Y^\an$.

Thus the claim follows from the standard fact that the weight spaces of $\cO(Y^\an)$ are finite-dimensional.
One easy proof of this fact goes as follows: Picking an equivariant algebraic embedding $j: Y \inj \bbC^m$, we can write $f \in \cO(Y^\an)$ locally around $j(o_Y) = 0$ as the restriction of a holomorphic function $g$. By averaging, we may assume that $g$ has the same weight as $f$. Then considering the Taylor expansion of $g$, we see that since $\xi$ acts on every coordinate of $\bbC^m$ with a \emph{positive} weight, $g$ must be a polynomial, and its degree is bounded by a constant times the weight of $f$.
\end{proof}

Combining \Cref{prop:energy-endpoint-variation} with \Cref{prop:complete-berndtsson}, we obtain the convexity of the Ding functional.
\begin{proposition}\label{prop:ding convexity}
     The Ding functional
     \[
         \cD_{\xi} := -\frac{1}{(2\pi)^n n! \bbW_X(\xi)} E_\xi + \cL
     \]
     is convex along a geodesic $\Phi_t$ in $\overline{\cH}_2^{\mathbb T}$ with $|\partial_t^\pm \Phi_t| \leq C(1 + r^2)$.
\end{proposition}

\section{Uniqueness of K\"ahler--Ricci shrinkers}\label{sec:uniqueness}

Suppose now that $(\omega_i, f_i)$, $i = 0, 1$, are K\"ahler--Ricci shrinkers on $X$ with the same soliton field $\xi$, and let $\phi_i$ be their canonical Hermitian metrics defined by
\[
    e^{-\phi_i} = e^{-f_i}\omega_i^n
    \p
\]
By \cref{lem:psh-weight-growth}, after weakening the constants, $\phi_0$ and $\phi_1$ lie in some $\cH_2(m, M)^\bbT$. Let $\Phi_t$, $0 \leq t \leq 1$, be the geodesic supplied by \cref{prop:general-geodesic}. At either endpoint, \cref{eq:general-hamiltonian} gives $h^{\phi_i} = f_i$.
In particular,
\begin{equation} \label{eq:shrinkerMAxi}
    \MA_\xi(\phi_i) = e^{-f_i} \omega_i^n = e^{-\phi_i}
    \p
\end{equation}

\begin{proposition}
    \label{prop:ding-affinity} 
    Along the geodesic $\Phi_t$ connecting two shrinker metrics $e^{-\phi_i}$, the functionals $D_\xi(\Phi_t), E_\xi(\Phi_t), \cL(\Phi_t)$ are constant.
\end{proposition}

\begin{proof}  
    The quadratic Lipschitz bound \cref{eq:general-geodesic-bound} and the integrability of $(1 + r^2)$ against $\xi$-weighted Monge--Ampère measures in \cref{thm:H2moments} allow us to compute the one-sided derivatives of $\cL$, by differentiating under the integral sign:
    \[
        \mathcal{L}'_+(0)=\exp(\mathcal L(\phi_0))\int_X\dot\Phi_0^+e^{-\phi_0},
        \qquad
        \mathcal{L}'_-(1)=\exp(\mathcal L(\phi_1))\int_X\dot\Phi_1^-e^{-\phi_1}
        \p
    \]
    At either shrinker endpoint, we have by \cref{eq:shrinkerMAxi}
    \begin{equation}\label{eq:endpoint-ma-equation}
        \int_X e^{-\phi_i}=(2\pi)^n n!\,\mathbb W_X(\xi),
    \end{equation}
    so in particular $\mathcal{L}(\phi_0)=\mathcal{L}(\phi_1)$. Thus combining \cref{eq:endpoint-ma-equation} with \cref{prop:energy-endpoint-variation} we see that the one-sided endpoint derivatives of the Ding functional satisfy.
    \[
        0 \leq (D_\xi)'_+(0) \leq (D_\xi)'_-(1) \leq 0 \p
    \]
    Hence $D_\xi(\Phi_t)$ is constant since $D_{\xi}$ is convex along a geodesic by \Cref{prop:ding convexity}.
    Since the weighted energy is affine, $\mathcal{L}(\Phi_t)$ is affine, and $\mathcal L(\phi_0)=\mathcal L(\phi_1)$ makes it constant. That forces $E_\xi$ to be constant by the definition of the Ding functional.
\end{proof}

\begin{theorem}[Metric uniqueness for a fixed soliton field]
    \label{thm:fixed-soliton-field-metric-uniqueness}
    Let $(\omega_i, f_i)$, $i = 0, 1$, be K\"ahler--Ricci shrinkers satisfying $J\nabla^{\omega_i}f_i = \xi$.  Then there is an  $\Psi\in \Aut_0(X,\xi)$, such that
    \[
        \Psi^*\omega_1 = \omega_0
        \p
    \]
\end{theorem}

\begin{proof}
    Let $\Phi_t$ be the geodesic supplied by \cref{prop:general-geodesic} between the two K\"ahler--Ricci shrinker potentials $e^{-\phi_i}=e^{-f_i}\omega_i^n$. By \cref{prop:ding-affinity}, $\mathcal L(\Phi_t)$ is constant, so the affine case of \cref{prop:complete-berndtsson} gives a $\bbT$-equivariant biholomorphism $F_1$ in the identity component satisfying
    \[
        F_1^*\phi_1=\phi_0,
        \qquad F_1^*\omega_1=\omega_0
        \p
    \]
    Taking $\Psi=F_1\in\Aut_0(X,\xi)$ proves the theorem.
\end{proof}

\begin{proof}[Proof of \Cref{thm:introduction-unique-krs}]
    This follows from \Cref{thm:introduction-unique-soliton-vector-field} and Theorem \ref{thm:fixed-soliton-field-metric-uniqueness}.
\end{proof}

\section{K-polystability}%
\label{sec:universal-k-polystability}
K-polystability of asymptotically conical (AC) shrinkers was established in \cite{CE}, following the strategy of Berman \cite{BermanKPolystability} and Collins--Sz\'ekelyhidi \cite{CollinsSzekelyhidiSasakiEinstein}. The proof in \cite{CE} builds on the AC geodesics and pluripotential theory developed in \cite{EAC}, but none of the arguments in \cite{CE} use the AC assumption in a fundamental way. Thus we can easily generalize the proof of polystability by using the geodesics from \cref{prop:general-geodesic}.

\subsection{Stability and basic set-up}
\label{subsec:kpolystability-conventions}

We write $\bbT$ for a compact torus containing the soliton field and
$\bbT^{\bbC}$ for its complexification.  Equivariance for $\bbT$ extends
to $\bbT^{\bbC}$ by holomorphic continuation along complexified
one-parameter subgroups.

We refer to \cite[Section 5]{SunZhangFano} and \cite[Section 2.1]{CE} for the notion of special test configurations of a polarized Fano fibration and K-polystability.  Here we only recall that for a special test configuration
$(\mathcal X\to\mathcal Y,\bbT,\eta)$  with central fiber  $\mathcal X_0$,  if $\eta$ is the generator covering the standard action on the
base, its Futaki invariant is
\begin{equation}\label{eq:kpolystability-futaki-convention}
 \operatorname{Fut}_\xi(\eta)
 =-\left.\frac{d}{ds}\right|_{s=0}
   \bbW_{\mathcal X_0}(\xi+s\eta)
    \p
\end{equation}
K-polystability means nonnegativity for every such configuration, with
equality only for an equivariant product $X\times\bbC$ whose generator is
\[
 \eta=\eta_0+\tau\partial_\tau,
 \qquad [\eta_0,\bbT]=0,
\]
where $\eta_0$ is a holomorphic vector field on $X$.

The general embedding theorem \cite{SunZhangFano,CE} gives integers
$\ell, N, M > 0$, a commuting action of
$\bbT^{\mathbb C}\times\bbC^*$, and equivariant embeddings
\begin{equation}\label{eq:kpolystability-ambient-embedding}
 \mathcal X\hookrightarrow
 \bbP^N\times\bbC^M\times\bbC,
 \qquad
 \mathcal Y\hookrightarrow\bbC^M\times\bbC,
 \qquad
 \mathcal O_{\bbP^N\times\bbC^M}(1)|_{\mathcal X}
 \simeq-\ell K_{\mathcal X/\bbC}
    \p
\end{equation}
The linearization is the canonical determinant linearization.  If
$S^1_\eta$ is the unit circle of the degenerating action, put
\[
 \bbK:=\bbT\times S^1_\eta,
\]
choose all ambient Hermitian and smoothing data $\bbK$-invariant, and let
\[
 \lambda_t:X\longrightarrow X_{e^{-t}}
\]
be the degenerating action at $e^{-t}$. As in \cite[Definition~4.6]{CE}, by restricting to \(\mathcal X\) the sum of a K\"ahler cone metric on the Euclidean factor, with Reeb vector field \(\xi\), and a Fubini–Study metric on the projective factor, we obtain a reference ambient metric. If we denote its additive \(\mathcal O(1)\)-weight on \(\mathcal X\) by \(\ell\psi_{\rm amb}\), then, under the identification
\(
\mathcal O(1)|_{\mathcal X}\simeq -\ell K_{\mathcal X/\bbC},
\)
the weight \(\psi_{\rm amb}\) is the induced anticanonical metric.  In the
canonical trivialization over $\bbC^*$, define
\begin{equation}\label{eq:kpolystability-ambient-flow-weight}
 \psi_t=\lambda_t^*\bigl(\psi_{\rm amb}|_{X_{e^{-t}}}\bigr)
    \p
\end{equation}It is direct to show that $\psi_t\in \cH_2^{\mathbb T}$ and therefore, by \cref{lem:flow-quadratic-control}, after changing
the ambient metric, we may assume
\begin{equation}\label{eq:kpolystability-initial-order}
 \psi_0 \leq \phi
    \p
\end{equation}

The positive $\xi$-weights on the affine factor in
\cref{eq:kpolystability-ambient-embedding} admit a proper  $\bbK$-invariant homogeneous
cone radius $r$ by the weighted-radius construction of
\cite[Section~4.2, formula~(4.1)]{CE} satisfying
\begin{equation}\label{eq:kpolystability-radius-homogeneity}
 r^2\circ\gamma_s=e^sr^2, \quad r^2\circ \lambda_t=e^{-t}r^2
    \p
\end{equation}
 We use
the same symbol for the restriction of $r$ to $X$.

\subsection{Geodesic ray, slope at infinity and K-polystability}

Let $\tau$ denote the test-configuration coordinate and put
\[
 \bbD=\{|\tau|<1\},\qquad
 \mathcal X_{\bbD}=\mathcal X\times_{\bbC}\bbD
    \p
\]
On the punctured disk we use the logarithmic time
\( t=-\log|\tau|.
\)
Thus ray time $t=0$ corresponds to the boundary fiber $X_1$, whereas the
central fiber $\mathcal X_0$ corresponds to $t=+\infty$.

The argument in \cite[Section 5 and Section 6]{CE},
with the normalization in \cref{eq:modified-ding}, applies in our general
setting and yields the following results.

\begin{proposition} \label{prop:kpolystability-geodesic-ray}
    There is a $\bbK$-invariant locally bounded psh metric $\Phi$ on
    $-K_{\mathcal X}|_{\mathcal X_{\bbD}}$ in the sense that $\ell\Phi$
    is a metric on $-\ell K_{\mathcal X}$, satisfying the following.
    \begin{enumerate}
        \item $\Phi$ solves the boundary value problem
            \begin{equation}\label{eq:kpolystability-ray-hcma}
                (\sqrt{-1}\partial\bar\partial\Phi)^{n+1}=0,
                \quad \Phi_0=\phi
                \p
            \end{equation}
        \item Under $\mathcal X^*\simeq X\times\bbC^*$, its slices $\Phi_t = \lambda_t^* \Phi|_{X_{e^{-t}}}$
            form a bounded family in one quadratic-growth class on every finite time interval.
        \item We have
        \begin{equation}\label{eq-uniform upper bound}
            \psi_\mathrm{amb}|_{\cX} \leq \Phi \leq \psi_\mathrm{amb}|_{\cX} + C(1 + r^2)
            \p
        \end{equation}
        Consequently $\psi_t \leq \Phi_t \leq \psi_t + C(1 + \lambda_t^* r^2)$.
        \item We have a locally uniform Lipschitz bound
            \begin{equation}\label{eq:kpolystability-ray-time-lipschitz}
                |\Phi_t - \Phi_s| \leq C_T(1+r^2)|t-s|,
                    \qquad 0 \leq s,t \leq T
                \p
            \end{equation}
    \end{enumerate}
\end{proposition}

\begin{proof}
    The construction in \cite[\S 4.4]{CE} does not use the AC hypothesis, so it works in our setting. The upper bound in \cref{eq-uniform upper bound} is stated in \cite[Lem.~4.18]{CE} only for $t \in [0, T]$, with a constant $C_T$ depending on $T$. However the argument in \cite[Prop.6.6, Claim~3]{EAC} shows that since $\Phi$ is a PSh metric defined on all of $\cX$ this can be upgraded to a bound near the central fiber, thus giving the upper bound claimed here.
\end{proof}

The proofs of \cite[Prop~5.3, Thm.5.5, Thm.~6.1]{CE} do not use the AC hypothesis. Thus we obtain
\begin{proposition} \label{prop:kpolystability-ding-minimum}
    Along the ambient subgeodesic ray $\psi_t$,
    \begin{equation}\label{eq:kpolystability-log-volume-slope}
        \lim_{t\to\infty}\frac{d}{dt}\mathcal{L}(\psi_t)=0, \quad \lim_{t\to\infty}
        \frac{d}{dt}E_{\xi}(\psi_t,\phi)
        =-2(2\pi)^nn!
        \operatorname{Fut}_\xi(\eta)
        \p
    \end{equation}
    and consequently
    \begin{enumerate}
        \item
            \(\lim_{t\to\infty}\frac{d}{dt}\cD_{\xi}(\psi_t)
            =\frac{2}{\bbW_X(\xi)}
            \operatorname{Fut}_\xi(\eta).\)
        \item The canonical shrinker weight minimizes the Ding functional in $\overline{\cH}_2^{\mathbb T}$: for any $\theta \in \overline{\cH}_2^{\mathbb T}$, we have
            \begin{equation}\label{eq:kpolystability-ding-minimum}
                \cD_{\xi}(\theta) \geq \cD_{\xi}(\phi)
                \p
            \end{equation}
        \item The associated polarized Fano fibration of a smooth K\"ahler--Ricci shrinker is K-semistable.
    \end{enumerate}
\end{proposition}

To prove K-polystability we first observe that as a consequence of Berndtsson's theorem \cref{prop:complete-berndtsson} we have
\begin{proposition}
    \label{prop:kpolystability-zero-slope-rigidity}
    If $\operatorname{Fut}_\xi(\eta)=0$, then the geodesic ray in
    \cref{prop:kpolystability-geodesic-ray} satisfies
    \begin{equation}\label{eq:kpolystability-functional-rigidity}
        \cD_{\xi}(\Phi_t)=\cD_{\xi}(\phi),
        \qquad
        \mathcal{L}(\Phi_t)=\mathcal{L}(\phi)
        \qquad(t\geq0)
        \p
    \end{equation}
    For every $t\geq0$ there is a $\bbT$-equivariant biholomorphism
    $F_t:X\to X$, commuting with $\xi$, such that
    \begin{equation}\label{eq:kpolystability-weight-rigidity}
        F_t^*\Phi_t=\phi
        \p
    \end{equation}
\end{proposition}
\begin{proof}
    See \cite[Lem~6.3]{EAC}.
\end{proof}

Assume for the remainder of the section that $\operatorname{Fut}_\xi(\eta)=0$, and choose the maps $F_t$ from \cref{prop:kpolystability-zero-slope-rigidity}.  Set
\begin{equation}\label{eq:kpolystability-Gt}
    G_t=\lambda_t\circ F_t:X\longrightarrow X_{e^{-t}}
    \p
\end{equation}
They are $\bbT$-equivariant, commute with $\gamma_s$, and satisfy
\begin{equation}\label{eq:kpolystability-Gt-weight-order}
    G_t^* \psi_\mathrm{amb} = F_t^*\psi_t \leq F_t^*\Phi_t=\phi
    \p
\end{equation}

The following lemma serves as our replacement of \cite[Cor.~4.3]{CE} in the proof of polystability.
\begin{lemma}
    There is a constant $C>0$, independent of $t$, such that for  $t\geq0$,
    \begin{equation}\label{eq:kpolystability-Gt-properness}
        C^{-1} r^2 \leq r^2 \circ G_t \leq C r^2
    \end{equation}
\end{lemma}
\begin{proof}
    Recall that by definition $\psi_\mathrm{amb} = \ell^{-1} \phi_\mathrm{FS} + r^2/2$ for $r \geq 99$. Thus pulling \cref{eq-uniform upper bound} by $G_t$ we obtain
    \[
        \frac1{\ell} G_t^* \phi_\mathrm{FS} + \frac12 r^2 \circ G_t \leq \phi \leq \frac1{\ell} G_t^* \phi_\mathrm{FS} + C(1 + r^2 \circ G_t)
    \]
    on the set $\{r \circ G_t \geq 99\}$.
    Since $\bbP^N$ is compact, we have $G_t^* \phi_\mathrm{FS} = \phi_\mathrm{FS} + O(t)$. Thus pulling the previous equation back by $\gamma_s$ we obtain
    \[
        \frac12 r^2 \circ G_t + O(t) \leq \phi - \phi_\mathrm{FS} \leq C r^2 \circ G_t + O(t)
        \quad \text{on $\{r \circ G_t \geq 99\}$}
        \p
    \]
    Pulling back by $\gamma_s^*$, again using that $\gamma_s^* \phi_\mathrm{FS} = \phi_\mathrm{FS} + O(s)$,  we obtain for $s$ large (depending on $t$)
    \[
        \frac12 e^s r^2 \circ G_t + O(t) \leq \gamma_s^* \phi - \phi_\mathrm{FS} + O(s) \leq C e^s r^2 \circ G_t + O(t)
    \]
    Taking the limit $s \to \infty$ on $\Sigma_0$ and using \cref{lem:psh-weight-growth} gives
    \begin{align}
        r^2 \circ G_t|_{\Sigma_0} & \leq 2M \\
        r^2 \circ G_t|_{\Sigma_0} & \geq C^{-1}m
    \end{align}
    This proves the claim by $X_0$-equivariance.
\end{proof}

Using \cref{eq:kpolystability-Gt-properness} and following the argument of \cite[Proposition 6.6]{CE}, we obtain the following result. For the reader’s convenience, we include the proof with the modifications needed for our setting.

\begin{proposition}
\label{prop:kpolystability-central-fiber-isomorphism}
If $\operatorname{Fut}_\xi(\eta)=0$, then a sequence $G_{t_j}$ converges locally uniformly to a $\bbT$-equivariant algebraic isomorphism
\[
    G_\infty:X\longrightarrow \mathcal X_0
    \p
\]
\end{proposition}

\begin{proof}
    \textit{Step 1: Existence of $G_\infty$.} To show that $G_t^* r^2 = r^2 \circ G_t$ is locally uniformly bounded we use \cref{eq:kpolystability-Gt-properness}. All other arguments in \textit{loc.cit.} work in the non-AC case.

    \textit{Step 2: $G_\infty: X \to X_0$ is a branched covering.}
    Claim 3 in \textit{loc.cit.} proves a $t$-uniform lower bound on $G_t^* r^2 = r^2 \circ G_t$ which is positive away from $\{r = 0\}$. Here this is provided by \cref{eq:kpolystability-Gt-properness}.

    Once we have established that the map $G_c : Y \to Y_0 \subseteq \bbC^M$ induced by $G_\infty$ on the cone $Y$ contracts no positive-dimensional analytic subsets we can finish this step as in \textit{loc.cit.}: If an analytic subset $Z \subseteq X$ is contracted by $G_\infty$, then $G_c$ must contract $\pi(Z)$, so $Z$ is contained in a single fiber of $\pi$. However $G_\infty^* \cO(1) = - \ell K_X$ and $-K_X$ is $\pi$-ample, which forces $Z$ to be zero-dimensional. The end of Step 2 in \textit{loc.cit.} mentions that $X \setminus E \cong Y \setminus o_Y$, but is actually unnecessary for the argument there.

    \textit{Step 3: $G_\infty$ is a biholomorphism.} The original proof applies without modification. Note that Step 2, or equivalently \eqref{eq:kpolystability-Gt-properness}, shows that $G_{\infty}$ is proper.
\end{proof}

\begin{proof}[Proof of \cref{thm:kpolystability}]
    By \Cref{prop:kpolystability-ding-minimum}, we know that $\operatorname{Fut}_\xi(\eta)\geq0$ for every
    $\bbT$-equivariant special test configuration.  In the equality case,
    \cref{prop:kpolystability-central-fiber-isomorphism} identifies the central
    and generic fibers equivariantly. The
    product criterion
    \cite[Proposition~3.6]{CE} (see also \cite[Proposition 1]{Donaldson2012Stability}) therefore identifies
    the family equivariantly with $X\times\bbC$.
\end{proof}

We can use \cref{thm:kpolystability} to exclude shrinkers on products with an unstable factor.
\begin{thm}
    Let $X_1$ and $X_2$ be compact or noncompact manifolds admitting a polarized Fano fibration structure. If $X_1 \times X_2$ admits a shrinker metric then both $X_1$ and $X_2$ are K-polystable.
\end{thm}

\begin{proof}
    Choose maximal tori $\bbT_i \subseteq \Aut(X_i)$ containing the given Reeb fields. We claim that $\bbT_1 \times \bbT_2$ is maximal in $\Aut(X_1 \times X_2)$. Indeed, let $V$ be an invariant holomorphic vector field on the product. On a coordinate patch $U \subseteq X_2$, write
    \[
        V = V_1 + \sum_{j = 1}^m a_j(x,z)\frac{\partial}{\partial z_j},
        \qquad V_1 \in H^0(X_1 \times U,\pr_1^*\cT_{X_1}).
    \]
    For fixed $z$, each $a_j(\cdot,z)$ is a $\bbT_1$-invariant holomorphic function on $X_1$, hence constant by the Reeb condition. Thus the $X_2$-component depends only on $X_2$. Interchanging the factors shows that $V$ splits into vector fields on $X_1$ and $X_2$. Any torus containing $\bbT_1 \times \bbT_2$ therefore acts separately on the factors, so maximality of $\bbT_1$ and $\bbT_2$ proves the claim.

    By \cite[Proposition~5.17]{EAC}, after applying a biholomorphism, we may assume that the soliton field has the form
    \[
        \xi = \xi_1 + \xi_2,
        \qquad \xi_i \in \Lie(\bbT_i).
    \]
    Positivity of the $\xi$-weights therefore implies that each $\xi_i$ is a Reeb field for $X_i \to Y_i$. We prove K-polystability with respect to these induced fields.

    Let $\mathcal X_1 \to \mathcal Y_1$ be a special test configuration of $(X_1 \to Y_1,\xi_1)$ with generator $\eta$. Taking its product with $X_2 \to Y_2$ gives a special test configuration of $(X_1 \times X_2,\xi)$. Since anticanonical sections and their weights factor under products, \cref{eq:kpolystability-futaki-convention} gives
    \[
        \operatorname{Fut}_{\xi}(\eta)
        = \bbW_{X_2}(\xi_2)\operatorname{Fut}_{\xi_1}(\eta)
        \geq 0,
    \]
    where the inequality follows from \cref{thm:kpolystability}. Since $\bbW_{X_2}(\xi_2) > 0$, this proves K-semistability of $(X_1,\xi_1)$.
    Suppose equality holds. Apply the equivariant equality statement in the proof of \cref{thm:kpolystability} with respect to the closures of the flows of $\xi_1$ and $\xi_2$. We obtain an equivariant product trivialization
    \[
        \Phi:\mathcal X_1 \times X_2 \longrightarrow X_1 \times X_2 \times \bbC.
    \]
    Its generator commutes with both factor flows. The preceding splitting argument shows that this generator is
    \[
        \eta_1 + \eta_2 + \tau\partial_\tau,
    \]
    with $\eta_i$ a holomorphic vector field on $X_i$. Normalize $\Phi$ to be the identity at $\tau = 1$. Over $\bbC^*$, equivariance implies that the first components of $\Phi$ and $\Phi^{-1}$ are independent of the $X_2$-variable. This independence extends across $\tau = 0$, so these components induce mutually inverse equivariant maps between $\mathcal X_1$ and $X_1 \times \bbC$. Thus the original test configuration is a product, proving K-polystability of $(X_1,\xi_1)$. Interchanging the factors proves the assertion for $X_2$.
\end{proof}

\begin{appendix}

\section{Regularization}

\begin{thm}\label{lem:global-regularization}
    Let $\Phi_t \in \overline{\cH}^\bbT_2(m, M), t \in [0, 1]$ be a jointly PSh bounded family. Then there exists a sequence of smooth joint PSh families $\Phi^\nu_t \in \cH_2^\bbT(m/2, M')$ for $t \in [1/\nu, 1 - 1/\nu]$ and $a, b > 0$ such that for $\chi = \gamma_a^* \psi_\mathrm{amb} + b$
    \begin{enumerate}
        \item $\Phi^\nu_t$ is pointwise decreasing to $\Phi_t$ and $\Phi_t^\nu \leq \chi$;
        \item $\Phi^\nu_t \equiv \chi - C_\nu$ on $\{r \geq \nu\}$ for some fixed smooth $\chi \in \cH_2^\bbT$. In particular, $\partial_t \Phi^\nu_t$ is compactly supported and $i\partial\bar\partial \Phi_t^\nu$ are complete K\"ahler metrics;
        \item If $|\partial^\pm_t \Phi_t| \leq L(1 + r^2)$ then $|\partial_t \Phi^\nu_t| \leq (L+1) (1 + r^2)$.
    \end{enumerate}
\end{thm}

The idea of the proof is as follows:
The B{\l}ocki--Ko{\l}odziej regularization gives us an approximation $\Psi^\nu$ on the compact set $\{r \leq \nu + 1\}$ for arbitrary $\nu$.
Then we pick a metric $\chi$ which grows faster than $\Phi$ as $r \to \infty$. This way, $\chi$ will also grow faster than all $\Psi^\nu$, so by picking appropriate $C_\nu \to \infty$ we can achieve that $\chi - C_\nu \geq \Psi^\nu$ for $r \geq \nu$. Then the maximum $\max \{\Psi^\nu, \chi - C_\nu\}$ will be our desired globally defined regularization.

\begin{proof}

    Choose a finite atlas $\cA = \{U_\alpha\}_\alpha$ of $X$ consisting of equivariant charts covering the end and finitely many precompact charts covering the remaining compact set. This also gives us an atlas for $X \times I$ with the obvious coordinate $t$ on $I$ and we will not differentiate between charts for $X$ and charts for $X \times I$. We pick relatively compact $V_\alpha \subseteq U_\alpha$ which still cover $X$, and functions $\eta_\alpha$ on $U_\alpha$ such that $\eta_\alpha|_{V_\alpha} = 0$ and $\eta_\alpha \equiv -1$ in a neighborhood of $\partial U_\alpha$.

    We denote the potentials of a metric $e^{-\psi}$ in a chart $\alpha \in \cA$ by $\psi_\alpha$.
    For $0 < \epsilon, \delta < 1$ we set
    \[
        \Psi_{\delta, \epsilon} = \frac{1}{1 + 2\epsilon} \widetilde{\max_\alpha}^\delta \left(\Phi_\alpha * \rho_\delta - \psi_{\mathrm{amb},\alpha} + \epsilon^2 \eta_\alpha \right) + \psi_\mathrm{amb} \qquad \text{on $X \times (\delta, 1-\delta)$}
        \c
    \]
    where $\rho$ is a mollifier in $\bbR \times \bbC^n$, supported on the unit ball and $\widetilde{\max}^\delta$ is the regularized max function $\max * \rho_\delta$.
    If we pick $\epsilon \leq \epsilon(\nu)$ small enough, then $\epsilon i \partial\bar{\partial} \eta_\alpha \geq -\omega_\mathrm{amb}$ on $\{r < \nu + 1\}$.
    By \cite[Eq.~(3)]{BlockiKolodziej}, $\Psi_{\delta, \epsilon}$ will be continuous on the precompact set $\{r < \nu + 1\} \times (\delta, 1-\delta)$ for $\delta \leq \delta(\epsilon)$ sufficiently small.
    $\Psi_{\delta, \epsilon}$ is jointly PSh in $t$ and $X$ and in fact we have
     \[
        i \partial_X\bar{\partial}_X \Psi_{\delta, \epsilon} \geq \left(\frac{-1 -\epsilon}{1 + 2 \epsilon} + 1\right) \omega_\mathrm{amb}
    \]
    on $\{r < \nu + 1\}$, so it is strictly PSh in the $X$-direction. It is not hard to see that $\Psi_{t, \delta, \epsilon} \in \cH_2(m/2, 2M)$ are a bounded family for $\epsilon$ and $\delta$ sufficiently small.

    Note that since $\Phi_t$ form a bounded family, we have $(\Phi_t)_\alpha ,\psi_{\mathrm{amb}, \alpha} = O(r^2)$ uniformly.
    Thus
    \begin{equation} \label{eq:Psiupper}
        \Psi_{t, \delta, \epsilon} - \psi_\mathrm{amb} \leq K (1 + h_\mathrm{AC})
    \end{equation}
    uniformly in $t, \delta, \epsilon$ for some $K > 1$. Recall from \cref{eq:hAC} that $h_\mathrm{AC}$ is a smooth function on $X$ which agrees with $r^2/2$ for large $r$.
    It is clear that $\Psi_{\delta, \epsilon}$ is decreasing as $\delta \to 0$, and from the quadratic bound above we see that $\Psi_{\delta, \epsilon} + 3 K \epsilon(1 + h_\mathrm{AC})$ is decreasing as $\epsilon \to 0$.

    By \Cref{lem:flow-quadratic-control} and taking $a, b \gg_K 1$ we obtain $\chi = \gamma_a^* \psi_\mathrm{amb} + b \in \cH_2^\bbT$ such that
    \begin{equation} \label{eq:Chilower}
        \psi_\mathrm{amb} + 3K (1 + h_\mathrm{AC}) \leq \chi
        \p
    \end{equation}
    Then pick $\epsilon_\nu \leq \min\left\{\epsilon(\nu), \frac{1}{4K} \right\}$ and $\delta_\nu \leq \delta(\epsilon_\nu)$, both decreasing to $0$ as $\nu \to \infty$ and set
    \[
        \Phi_t^\nu = \widetilde{\max}^1 \left( \Psi_{t, \delta_\nu, \epsilon_\nu} + 4 K \epsilon_\nu(1 + h_\mathrm{AC}), \chi - K(1 + \nu^2/2)\right)
        \p
    \]
    By \cref{eq:Psiupper,eq:Chilower} this agrees with $\chi - K(1 + \nu^2/2)$ for $r \geq \nu$ and thus defines a smooth PSh metric.

    Since convolution commutes with the $t$-derivative, we easily see that $|\partial_t \Psi_{\delta, \epsilon}| \leq (L+1)(1 + r^2)$ for $\delta \ll 1$, which then gives $|\partial_t \Phi_t^\nu| \leq (L+1)(1 + r^2)$
    $\bbT$-invariance can be achieved by averaging over $\bbT$, just as in \cite[Thm.~A.1]{EAC}.
\end{proof}

\end{appendix}

\linespread{1.0}
\bibliographystyle{alpha}
\bibliography{references}

@article{LWII,
  author  = {Li, Yu and Wang, Bing},
  title   = {Heat kernel on {Ricci} shrinkers ({II})},
  journal = {Acta Math. Sci. Ser. B (Engl. Ed.)},
  volume  = {44},
  number  = {5},
  pages   = {1639--1695},
  year    = {2024},
  doi     = {10.1007/s10473-024-0502-7}
}

@misc{Zhang2026,
      title={Some uniform estimates for K\"ahler--Ricci Shrinkers}, 
      author={Junsheng Zhang},
      year={2026},
      eprint={2608.30779},
      archivePrefix={arXiv},
      primaryClass={math.DG},
      url={https://arxiv.org/abs/2608.30779}, 
}

@misc{CD2026,
      title={Uniqueness of the soliton vector field of a shrinking gradient K\"ahler-Ricci soliton}, 
      author={Ronan J. Conlon and Alix Deruelle},
      year={2026},
      eprint={2610.02164},
      archivePrefix={arXiv},
      primaryClass={math.DG},
      url={https://arxiv.org/abs/2610.02164}, 
}

@book{helgason1979differential,
  title={Differential geometry, Lie groups, and symmetric spaces},
  author={Helgason, Sigurdur},
  volume={80},
  year={1979},
  publisher={Academic press}
}

@article{Berndtsson2015,
  author  = {Berndtsson, Bo},
  title   = {A {B}runn--{M}inkowski type inequality for {F}ano manifolds
             and some uniqueness theorems in {K}\"ahler geometry},
  journal = {Invent. Math.},
  volume  = {200},
  number  = {1},
  pages   = {149--200},
  year    = {2015},
  doi     = {10.1007/s00222-014-0532-1},
}

@article{BEGZ,
  author  = {Boucksom, S\'ebastien and Eyssidieux, Philippe and Guedj, Vincent
             and Zeriahi, Ahmed},
  title   = {Monge--{A}mp\`ere equations in big cohomology classes},
  journal = {Acta Math.},
  volume  = {205},
  number  = {2},
  year    = {2010},
  pages   = {199--262},
  eprint  = {0812.3674},
  archivePrefix = {arXiv},
  primaryClass  = {math.CV},
  mrnumber = {2746347},
  zbl      = {1213.32025},
}

@article{Donaldson2012Stability,
  author        = {Donaldson, Simon K.},
  title         = {Stability, birational transformations and the
                   {K{\"a}hler--Einstein} problem},
  journal       = {Surveys in Differential Geometry},
  volume        = {17},
  number        = {1},
  pages         = {203--228},
  year          = {2012},
  publisher     = {International Press},
  doi           = {10.4310/SDG.2012.v17.n1.a5},
  eprint        = {1007.4220},
  archivePrefix = {arXiv},
  primaryClass  = {math.DG}
}

@article{iwasawa1949,
  title={On some types of topological groups},
  author={Iwasawa, Kenkichi},
  journal={Annals of mathematics},
  volume={50},
  number={3},
  pages={507--558},
  year={1949},
  publisher={JSTOR}
}

@book{kobayashi2012transformation,
  title={Transformation groups in differential geometry},
  author={Kobayashi, Shoshichi},
  year={2012},
  publisher={Springer Science \& Business Media}
}

@article{TianZhuUniqueness,
  author  = {Tian, Gang and Zhu, Xiaohua},
  title   = {Uniqueness of {K}{\"a}hler--{R}icci solitons},
  journal = {Acta Mathematica},
  volume  = {184},
  number  = {2},
  pages   = {271--305},
  year    = {2000}
}

@article{ChenScalar,
  author  = {Chen, Bing-Long},
  title   = {Strong uniqueness of the {R}icci flow},
  journal = {Journal of Differential Geometry},
  volume  = {82},
  number  = {2},
  pages   = {363--382},
  year    = {2009}
}

@article{CDS,
  author  = {Conlon, Ronan J. and Deruelle, Alix and Sun, Song},
  title   = {Classification results for expanding and shrinking gradient {K}{\"a}hler--{R}icci solitons},
  journal = {Geometry \& Topology},
  volume  = {28},
  number  = {1},
  pages   = {267--351},
  year    = {2024},
  doi     = {10.2140/gt.2024.28.267}
}

@article{CDSIII,
  title={K{\"a}hler-{E}instein metrics on {F}ano manifolds. {III}: {L}imits as cone angle approaches $2\pi$ and completion of the main proof},
  author={Chen, Xiuxiong and Donaldson, Simon and Sun, Song},
  journal={Journal of the American Mathematical Society},
  volume={28},
  number={1},
  pages={235--278},
  year={2015}
}

@misc{EAC,
  author        = {Esparza, Carlos},
  title         = {Uniqueness of asymptotically conical shrinking gradient {K}{\"a}hler--{R}icci solitons},
  year          = {2025},
  eprint        = {2502.13521},
  archiveprefix = {arXiv},
  primaryclass  = {math.DG},
  doi           = {10.48550/arXiv.2502.13521}
}

@misc{CE,
  author        = {Cifarelli, Charles and Esparza, Carlos},
  title         = {K-polystability of Asymptotically Conical {K}{\"a}hler--{R}icci Shrinkers},
  year          = {2025},
  eprint        = {2512.03323},
  archiveprefix = {arXiv},
  primaryclass  = {math.DG},
  doi           = {10.48550/arXiv.2512.03323}
}

@misc{LiZhang2026,
  author        = {Li, Yu and Zhang, Junsheng},
  title         = {{Gromov--Hausdorff Limits of Noncollapsed K{\"a}hler--Ricci Flows and the Geometry of Ricci Shrinkers}},
  note          = {Preprint; arXiv:2607.25644 [math.DG]},
  year          = {2026},
  eprint        = {2607.25644},
  archiveprefix = {arXiv},
  primaryclass  = {math.DG}
}

@article{BermanKPolystability,
  author  = {Berman, Robert J.},
  title   = {{K}-polystability of {$\mathbb{Q}$}-{F}ano varieties admitting {K}{\"a}hler--{E}instein metrics},
  journal = {Inventiones Mathematicae},
  volume  = {203},
  number  = {3},
  pages   = {973--1025},
  year    = {2016},
  doi     = {10.1007/s00222-015-0607-7}
}

@article{CollinsSzekelyhidiSasakiEinstein,
  author  = {Collins, Tristan C. and Sz{\'e}kelyhidi, G{\'a}bor},
  title   = {Sasaki--{E}instein metrics and {K}-stability},
  journal = {Geometry \& Topology},
  volume  = {23},
  number  = {3},
  pages   = {1339--1413},
  year    = {2019},
  doi     = {10.2140/gt.2019.23.1339}
}

@misc{SunZhangFano,
  author        = {Sun, Song and Zhang, Junsheng},
  title         = {{K}{\"a}hler--{R}icci shrinkers and {F}ano fibrations},
  year          = {2024},
  eprint        = {2410.09661},
  archiveprefix = {arXiv}
}

@article{BlockiKolodziej,
  author  = {B{\l}ocki, Zbigniew and Ko{\l}odziej, S{\l}awomir},
  title   = {On regularization of plurisubharmonic functions on manifolds},
  journal = {Proceedings of the American Mathematical Society},
  volume  = {135},
  number  = {7},
  pages   = {2089--2093},
  year    = {2007},
  doi     = {10.1090/S0002-9939-07-08858-2}
}

@misc{BermanWittNystrom,
  author        = {Berman, Robert J. and Witt Nystr{\"o}m, David},
  title         = {Complex optimal transport and the pluripotential theory of {K}{\"a}hler--{R}icci solitons},
  year          = {2014},
  eprint        = {1401.8264},
  archiveprefix = {arXiv}
}

@book{GuedjZeriahi,
  author    = {Guedj, Vincent and Zeriahi, Ahmed},
  title     = {Degenerate Complex {M}onge--{A}mp{\`e}re Equations},
  series    = {EMS Tracts in Mathematics},
  volume    = {26},
  publisher = {European Mathematical Society},
  address   = {Z{\"u}rich},
  year      = {2017},
  doi       = {10.4171/167}
}
\end{document}